\documentclass[12pt]{article}
\usepackage[colorlinks=true,citecolor=black,linkcolor=black,urlcolor=blue,backref=page]{hyperref}
\usepackage{e-jc}
\usepackage[font=footnotesize]{caption} 
\usepackage{amsmath, amsthm, amssymb, bm}
\usepackage{tikz,tkz-graph}

\usetikzlibrary{arrows,shapes,snakes,automata,backgrounds,petri,patterns}

\theoremstyle{plain}
\newtheorem*{wegner}{Wegner's Planar Graph Conjecture}
\newtheorem*{weg-conj}{Wegner's Conjecture}

\newtheorem*{SLCC}{Square List Coloring Conjecture}
\newtheorem*{TLCC}{Total List Coloring Conjecture}

\newtheorem*{L21}{$\bm{L(2,1)}$-Labeling Conjecture}

\theoremstyle{definition}

\theoremstyle{remark}

\newcommand{\fancy}[1]{\mathcal{#1}}

\newcommand{\G}{\fancy{G}}
\newcommand{\chil}{\chi_{\ell}}
\newcommand{\chip}{\chi_{p}}
\newcommand{\chiAT}{\chi_{\textrm{AT}}}
\newcommand{\col}{\chi_{\textrm{col}}}
\newcommand{\chif}{\chi_{f}}
\newcommand{\vph}{\varphi}

\def\floor#1{\lfloor#1\rfloor}
\def\Floor#1{\left\lfloor#1\right\rfloor}
\def\ceil#1{\lceil#1\rceil}
\def\Ceil#1{\left\lceil#1\right\rceil}
\def\chil{\chi_{\ell}}
\def\mad{{\rm mad}}
\def\ad{{\rm ad}}

\def\dist{{\rm dist}}

\def\DeltaG{\Delta}
\def\erdos{Erd\H{o}s}
\def\lovasz{Lov\'{a}sz}
\def\nesetril{Ne\v{s}et\v{r}il}
\def\kral{Kr\'{a}l'}

\newcommand{\GG}[1]{}

\usepackage[immediate]{silence}
\usepackage{sectsty} 
\DeactivateWarningFilters[temp] 
\allsectionsfont{\sffamily}

\RequirePackage{marginnote,hyperref}
\newcommand{\aside}[1]{\marginnote{\scriptsize{#1}}[0cm]}
\newcommand{\aaside}[2]{\marginnote{\scriptsize{#1}}[#2]}
\newcommand\Emph[1]{\emph{#1}\aside{#1}}
\newcommand\EmphE[2]{\emph{#1}\aaside{#1}{#2}}

\mydateline{Dec 1, 2021}{Mar 7, 2023}{~\\Version 1 published: Apr 21, 2023\\ Version 2 published: Apr 14, 2026}
\MSC{05C15, 05C10, 05C76}
\Copyright{The author. Released under the CC BY-ND license (International 4.0).}

\author{Daniel W. Cranston\thanks{Department of Mathematics,
William and Mary, Williamsburg, Virginia, USA; \texttt{dcransto@gmail.com}.}
}

\def\volno{\!\!} %{30(2)}
\def\volyear{2026}
\def\papno{DS25}
\def\papid{10898}

\title{Coloring, List Coloring, and Painting Squares of Graphs (and other
related problems)}

\begin{document}
\maketitle
\vspace{-.45in}

\begin{abstract}
We survey work on coloring, list coloring, and painting squares of graphs;
in particular, we consider strong edge-coloring.
We focus primarily on planar graphs and other sparse classes of graphs. 
\end{abstract}
\renewcommand{\baselinestretch}{0.915}\normalsize
\tableofcontents
\renewcommand{\baselinestretch}{1.0}\normalsize

\section{Introduction}
The \Emph{square} $G^2$ of a graph $G$ is formed from $G$ by adding an edge
between each pair of vertices at distance 2\ in $G$.  Over the past 25 years, a
remarkable amount of work has focused on bounding the chromatic number of
squares of graphs.  
This problem was first studied by Kramer and
Kramer~\cite{KramerKramer-old2,KramerKramer-old1}
and much of its current popularity is due to Wegner's conjecture on coloring
squares of planar graphs~\cite{Wegner77}, which we consider in
Section~\ref{sec:wegner}.

The goal of this survey is to present major open questions on
coloring squares, as well as many current best results.  In addition, we
provide history of most problems, often including partial results, as well
as proofs that are particularly enlightening or enjoyable.  Our aim is not
to be exhaustive (or exhausting), but rather to help
the reader get the lay of the land.
In the spirit of open access, in
addition to standard journal citations, whenever possible we also provide links
to freely accessible, often preliminary, versions of the papers. Typically,
these are on preprint servers such as the arXiv.

For many results in this survey, we will not say much about the proofs. 
Likewise, we de-emphasize questions of algorithms and complexity.  So 
it is useful to say a bit now.  

The most common technique for proving coloring bounds is coloring 
greedily in some good vertex order.  More often than not, this order is constructed
using the discharging method.  A standard reference for discharging is
\emph{An Introduction to the Discharging Method via Graph
Coloring}~\cite{CranstonW-guide}, by the author and West.  
Most existence proofs using the discharging method naturally yield  
efficient coloring algorithms.
More details are given in \cite[Section 6]{CranstonK08}.

The second most common
technique employed here is the probabilistic method.  To learn about this
approach, we recommend the excellent monograph \emph{Graph Colouring and the Probabilistic
Method}~\cite{MolloyR-GCPM}, by Molloy and Reed.  At the heart of many
probabilistic coloring proofs is the {\lovasz} Local Lemma (LLL), which was originally
proved non-constructively.  Much work has focused on giving constructive proofs
of LLL (see~\cite{MoserT10} and the references cited there), so now these
probabilistic proofs often also yield efficient algorithms.
For instance, entropy compression is a way to show
that a randomized algorithm runs in expected polynomial time; for example,
see~\cite{Molloy-EC} and~\cite{EP-tutorial}.
\bigskip

%\subsection{Definitions and Notation}
We assume standard graph theory terminology, as in Diestel~\cite{DiestelBook},
West~\cite{West-IGT}, and Bondy \& Murty~\cite{BondyMurty-book}. However,
notation for many types of coloring varies among authors, so we include ours
below.  For completeness, we also include some definitions.
Most of our choices are standard, but we have a few exceptions,
particularly in the penultimate paragraph of this subsection.
The reader should feel free to skip ahead to Section~\ref{gen-graphs:sec} and
only return to this section as needed.

A \EmphE{proper vertex coloring}{-2.5mm} of a graph $G$ assigns to each vertex
$v$ of $G$ a color, such that any two adjacent vertices get
distinct colors.  A \EmphE{$k$-coloring}{3mm} is a proper vertex coloring that uses at
most $k$ colors.  The \emph{chromatic number} \aaside{chromatic number,
$\chi$}{2.5mm}
of $G$, denoted $\chi(G)$, is the
smallest $k$ such that $G$ has a $k$-coloring.  
When our meaning is clear, we often shorten $\chi(G)$ to $\chi$, and similarly
for other parameters, such as $\chil$.
A \EmphE{$k$-list
assignment}{1mm} $L$ for a graph $G$ assigns to each vertex $v$ of $G$ a list of
allowable colors, $L(v)$, such that $|L(v)|=k$ for all $v$.  A \EmphE{proper
$L$-coloring}{0mm} is a proper vertex coloring $\varphi$ of $G$ such that each vertex gets a
color from its list, i.e., $\varphi(v)\in L(v)$ for all $v$.  The \emph{list
chromatic number}\aaside{list chromatic number, $\chil$}{3.5mm}, $\chil(G)$, is the smallest $k$ such that $G$ has a proper
$L$-coloring for every $k$-list assignment $L$.

The \EmphE{$k$-painting game}{5mm} is played between two players, \emph{Lister}
and \emph{Painter}\aaside{Lister, Painter}{14mm}.  Fix a graph $G$.  On each
round, Lister lists some non-empty subset of the unpainted vertices, and Painter
colors (or ``paints'') some independent subset of these.  If Lister
lists any vertex $k$ times without Painter painting it, then Lister wins;
otherwise, Painter wins.  The \EmphE{paint number}{8mm} of $G$, denoted
$\chip(G)$, also called the \emph{online list chromatic number} of $G$, is the
smallest $k$ for which Painter can always win the $k$-painting game on $G$.
The term ``online'' derives from an \emph{online algorithm},
which begins outputting a solution (a list coloring) before the entire
input (the list assignment) is known.
%\newpage

For an orientation $\vec{D}$ of a graph $G$, let $EE(\vec{D})$ (resp.
$EO(\vec{D})$) denote the set of spanning Eulerian subgraphs of $G$ with an
even (resp.\ odd) number of edges.  An \EmphE{Alon--Tarsi orientation}{-9.5mm} of $G$ is
one for which $|EE(\vec{D})|\ne |EO(\vec{D})|$; for convenience, we agree that
$EE$ contains the spanning subgraph with no edges.  The \emph{Alon--Tarsi
number} of $G$, denoted $\chiAT(G)$, \aaside{Alon--Tarsi number, $\chiAT$}{-9.5mm} 
is defined by $\chiAT(G):=1+\min_{\vec{D}}\max_{v\in
V(G)}d^+_{\vec{D}}(v)$, where the minimum is taken over all Alon--Tarsi
orientations of $G$.  The \emph{coloring number}\aaside{coloring number, 
$\col$}{0mm} of $G$, denoted $\col(G)$, is
defined by $\col(G) := 1+\max_{H\subseteq G}\delta(H)$; here $\delta(H)$ is the
minimum degree of $H$.
It is useful to note that every graph $G$ satisfies 
$$\chi(G)\le\chil(G)\le\chip(G)\le\chiAT(G)\le\col(G)\le \Delta(G)+1.$$

The first two inequalities follow directly from the definitions, if the lists
are identical or they are revealed progressively.  The last
inequality is trivial.  Alon and Tarsi~\cite{AlonT} proved that
$\chil(G)\le\chiAT(G)$ and Schauz~\cite{Schauz-AT} strengthened this to
$\chip(G)\le \chiAT(G)$.  We can see that $\chiAT(G)\le \col(G)$ as follows.  
Form the order $v_1,\ldots, v_n$ by starting from $G$ and repeatedly taking
$v_i$ to be a vertex of minimum degree in the remaining subgraph, and then
deleting $v_i$.  Now form $\vec{D}$ by orienting each edge as $\overrightarrow{v_iv_j}$
where $i<j$.  Since this orientation $\vec{D}$ is acyclic, it has
$|EE(\vec{D})|=1$ and $|EO(\vec{D})|=0$.  This proves that $\chiAT(G)\le\col(G)$.
(It has been shown~\cite{ERT,DGK} that for the graph $K_{n,n}$, as $n\to \infty$,
each of the differences $\chil(G)-\chi(G)$, $\chip(G)-\chil(G)$, $\chiAT(G)-\chip(G)$,
$\col(G)-\chiAT(G)$, and $\Delta(G)+1-\col(G)$ grows arbitrarily large; alternatively, 
see~\cite[Section 1.2]{GCM}.)

We denote the chromatic number of the square of $G$ by
$\chi^2(G)$; more generally, $\chi^d(G)$\aside{$\chi^2$, $\chi^d$} denotes the
chromatic number of the $d^{\textrm{th}}$ power of $G$, denoted \Emph{$G^d$}
(which is formed from $G$ by adding an edge between each pair of vertices at
distance no more than $d$).  A \EmphE{strong edge-coloring}{-5.5mm} of
$G$ is a coloring of the
square of the line graph of $G$.  The strong edge-chromatic number is denoted $\chi^s(G)$.
A \EmphE{total coloring}{1mm} of $G$ colors vertices and edges so that elements get
distinct colors whenever they are incident or adjacent.  
The \Emph{total chromatic number} is denoted $\chi''(G)$\aaside{$\chi''$}{11mm}.  
An $L(p,q)$-labeling assigns labels (positive integers) to vertices of $G$ so
that vertices $v$ and $w$ have labels differing by at least $p$ (resp.~at least
$q$) whenever $v$ and $w$ are adjacent (resp.\ distance 2).  The
\EmphE{span}{8mm} of an $L(p,q)$-labeling is the difference between the largest
and smallest labels, and the $(p,q)$-span of a
graph $G$ is the minimum span of an $L(p,q)$-labeling of $G$; this is denoted
$\lambda^{p,q}(G)$\aaside{$\lambda^{p,q}$}{5mm}.  For most of these parameters, it
makes sense to define analogues for list coloring, paint number, Alon--Tarsi
number, and coloring number.  For $G^2$, we denote these analogues as $\chi^2(G)$,
$\chil^2(G)$, $\chip^2(G)$, $\chiAT^2(G)$, and $\col^2(G)$.  
Our focus in this survey, which reflects the focus in the literature, is on
$\chi^2(G)$ and $\chil^2(G)$.  However, often proofs of bounds for these parameters
actually prove the same bound for $\col^2(G)$, which we mention when applicable.

For a graph $G$, we write $\Delta(G)$, $\omega(G)$, and $g(G)$ for its maximum
degree, its clique number\aside{$\Delta$, $\omega$, $g$}, and its girth
(length of its shortest cycle).  When $G$ is clear from context,
we often shorten these to $\Delta$, $\omega$, and $g$.
The \Emph{maximum average degree}, $\mad(G)$\aaside{$\mad$}{11mm}, of a graph $G$ is
$\max_{H\subseteq G,|V(H)|\ge 1}\frac{2|E(H)|}{|V(H)|}$.  When $G$ is planar, 
Euler's formula gives $\mad(G)<\frac{2g}{g-2}$; see Lemma~\ref{lem:folklore}.  
As we will see in Section~\ref{sec:mad},
many results initially proved for planar graphs with sufficiently large girth
actually hold for the larger class of graphs with bounded maximum average degree.

\section{Graphs in General: Wegner's Conjecture}
\label{gen-graphs:sec}
For every graph $G$, clearly $\chi(G)\ge\omega(G)$.  In
general, this bound can be quite bad.  Mycielski constructed%
\footnote{Given a graph $G$ with
$V(G)=\{v_1,\ldots,v_n\}$, we define the Mycielskian $M(G)$ as follows.  Let
$V(M(G)):=\{v_1,\ldots,v_n,w_1,\ldots,w_n,x\}$.  Let
$E(M(G)):=E(G)\cup\{v_iw_j~|~v_iv_j\in E(G)\}\cup\{w_ix~|~i\in\{1,\ldots,n\}\}$.
It is easy to see that $\chi(M(G))=\chi(G)+1$ and that if $G$ is
triangle-free, then so is $M(G)$.} 
triangle-free graphs with arbitrarily high chromatic number,
and \erdos\ proved the existence of
a graph with both chromatic number and girth arbitrarily high.
However, both constructions have maximum degree much higher than 
chromatic number.  In contrast, for a given maximum degree, the graph with
maximum chromatic number is the complete graph. Wegner~\cite{Wegner77} believed
that something similar is true when we consider graphs raised
to a fixed power.

\begin{weg-conj}[\cite{Wegner77}] %p. 8 in his preprint
\label{wegner-conj}
For all integers $k\ge 1$ and $D\ge 3$, let $\chi^k(D)$ and $\omega^k(D)$
denote, respectively, the maximums over all graphs $G$ with $\Delta\le D$ of
$\chi(G^k)$ and $\omega(G^k)$.  For all $k$ and $D$ we have
$\chi^k(D)=\omega^k(D)$.
\end{weg-conj}

This conjecture is remarkably wide-ranging.  Wegner writes: ``one cannot expect
a general answer, but it would be interesting to settle some cases.''  
The restriction to $D\ge 3$ is because the case $D=1$ is trivial, and the case
$D=2$ is quite easy.  (At the end of Section~\ref{sec:planar-girth}, we
determine $\chi^2(C_n)$, $\chil^2(C_n)$, and $\chiAT^2(C_n)$ for each cycle $C_n$.)
The case $k=1$ is also immediate, by considering greedy
coloring and the complete graph $K_{D+1}$.  So 
we start with $k=2$ and $D\ge 3$.

Greedy coloring shows that $\chi(G)\le\Delta+1$ for all $G$.  Since $G^2$ has
maximum degree at most $\Delta^2$, we immediately have $\chi^2(G)\le
\Delta^2+1$.  Applying Brooks' Theorem to $G^2$ shows that equality holds for a
connected graph $G$ only if $G^2=K_{\Delta^2+1}$.  Hoffman and
Singleton~\cite{HoffmanS60} famously used
linear algebra to show this is possible only if $\Delta\in\{2,3,7,57\}$.  
This proof is also presented in~\cite[Chapter~14]{33miniatures}.
The unique realizations for $\Delta=2$ and $\Delta=3$ are the 5-cycle and the
Petersen graph.  For $\Delta=7$, the only realization is the Hoffman--Singleton
graph~\cite{HoffmanS60}.  For $\Delta=57$, the question of whether any
realization exists remains a major open problem.  ({Makhnev~\cite{makhnev}
claimed to disprove the existence of any realization when $\Delta=57$. 
However, Faber~\cite{faber} claimed to refute this proof.})

\erdos, Fajtlowicz, and Hoffman~\cite{ErdosFH80} used the same approach to show
that when $\Delta\ge 3$, no graph $G$ has $G^2=K_{\Delta^2}$. 
Elspas~\cite{Elspas64} constructed graphs $G$ for each $\Delta\in\{4,5\}$ such that
$G^2=K_{\Delta^2-1}$; see Figure~\ref{Elspas-fig}.  Thus, to prove Wegner's
Conjecture for $k=2$ and $\Delta\in\{4,5\}$, it suffices to prove that
$\chi^2(G)\le\Delta^2-1$ whenever $\Delta\in\{4,5\}$.  As a first step, it is
useful to prove that $\omega(G^2)\le\Delta^2-1$.
Fortunately, the above result of \erdos\ et~al.\ yields the following lemma as
an easy corollary; this was first noted in~\cite{CranstonK08}.

\begin{lemma}[\cite{CranstonK08}]
\label{lem:no-big-cliques}
If $G$ is connected, $\Delta\ge 3$, and $G^2\ne K_{\Delta^2+1}$, then
$\omega(G^2)\le \Delta^2-1$.
\end{lemma}
\begin{proof}
Suppose to the contrary that $G^2\ne K_{\Delta^2+1}$ and $\omega(G^2)\ge \Delta^2$.
Let $S$ be a maximum clique in $G^2$. \erdos\
et~al.~\cite{ErdosFH80} showed that $S$ must be a proper subset of $V(G)$.  Since
$G$ is connected, there exist adjacent vertices $v$ and $w$ with $v\in S$ and
$w\notin S$.  Note that $d_{G^2}(v)\le \Delta^2$.  Since $|S|\ge \Delta^2$ and
$w\notin S$, the set $S$ contains every neighbor of $v$ in $G^2$ other than
$w$.  In particular, $S$ contains every neighbor of $w$ in $G$.  Now repeating
this argument for each neighbor of $w$ in $G$, we conclude that $S$ contains
every vertex at distance at most 2 from $w$ in $G$, i.e., every neighbor of $w$
in $G^2$.  Thus, $S\cup\{w\}$ is a clique in $G^2$ of size
$|S|+1$, which contradicts our choice of $S$ as maximum.
\end{proof}

\begin{figure}[!t]
\centering

\begin{tikzpicture}[scale = 3.8]
\tikzstyle{VertexStyle}=[shape = circle,	
minimum size = 5pt, inner sep = 2pt,
                                 draw]
\Vertex[x = 1.5, y = -0.0499999523162842, L = \tiny {}]{v0}
\Vertex[x = 1.60000002384186, y = 0.100000023841858, L = \tiny {}]{v1}
\Vertex[x = 1.35000002384186, y = 0.199999988079071, L = \tiny {}]{v2}
\Vertex[x = 1.54999995231628, y = 0.199999988079071, L = \tiny {}]{v3}
\Vertex[x = 1.5, y = 0.449999988079071, L = \tiny {}]{v4}
\Vertex[x = 1.60000002384186, y = 0.300000011920929, L = \tiny {}]{v5}
\Vertex[x = 1.79999995231628, y = 0.449999988079071, L = \tiny {}]{v6}
\Vertex[x = 1.70000004768372, y = 0.300000011920929, L = \tiny {}]{v7}
\Vertex[x = 1.95000004768372, y = 0.199999988079071, L = \tiny {}]{v8}
\Vertex[x = 1.75, y = 0.199999988079071, L = \tiny {}]{v9}
\Vertex[x = 1.79999995231628, y = -0.0499999523162842, L = \tiny {}]{v10}
\Vertex[x = 1.70000004768372, y = 0.100000023841858, L = \tiny {}]{v11}
\Vertex[x = 2.25, y = -0.149999976158142, L = \tiny {}]{v12}
\Vertex[x = 1.04999995231628, y = -0.149999976158142, L = \tiny {}]{v13}
\Vertex[x = 1.64999997615814, y = 0.849999994039536, L = \tiny {}]{v14}
\Edge[](v0)(v1)
\Edge[](v0)(v3)
\Edge[](v0)(v11)
\Edge[](v0)(v12)
\Edge[](v1)(v2)
\Edge[](v1)(v7)
\Edge[](v1)(v10)
\Edge[](v2)(v3)
\Edge[](v2)(v5)
\Edge[](v2)(v14)
\Edge[](v3)(v4)
\Edge[](v3)(v9)
\Edge[](v4)(v5)
\Edge[](v4)(v7)
\Edge[](v4)(v13)
\Edge[](v5)(v6)
\Edge[](v5)(v11)
\Edge[](v6)(v7)
\Edge[](v6)(v9)
\Edge[](v6)(v12)
\Edge[](v7)(v8)
\Edge[](v8)(v9)
\Edge[](v8)(v11)
\Edge[](v8)(v14)
\Edge[](v9)(v10)
\Edge[](v10)(v11)
\Edge[](v10)(v13)
\Edge[](v12)(v13)
\Edge[](v12)(v14)
\Edge[](v13)(v14)
\end{tikzpicture}
\hspace{.5in}
\begin{tikzpicture}[scale = 4]
\tikzstyle{VertexStyle}=[shape = circle, minimum size = 5pt, inner sep = 2pt, draw]
%\begin{scope}[xshift=4in]
\Vertex[x = 0.800000011920929, y = 0.75, L = \tiny {}]{v0}
\Vertex[x = 0.800000011920929, y = 0.149999976158142, L = \tiny {}]{v1}
\Vertex[x = 1.79999995231628, y = 0.75, L = \tiny {}]{v2}
\Vertex[x = 1.79999995231628, y = 0.149999976158142, L = \tiny {}]{v3}
\Vertex[x = 1.29999995231628, y = 0.149999976158142, L = \tiny {}]{v4}
\Vertex[x = 1.29999995231628, y = 0.75, L = \tiny {}]{v5}
\Vertex[x = 0.649999976158142, y = 0.399999976158142, L = \tiny {}]{v6}
\Vertex[x = 0.949999988079071, y = 0.399999976158142, L = \tiny {}]{v7}
\Vertex[x = 0.800000011920929, y = 0.399999976158142, L = \tiny {}]{v8}
\Vertex[x = 1.95000004768372, y = 0.399999976158142, L = \tiny {}]{v9}
\Vertex[x = 1.79999995231628, y = 0.399999976158142, L = \tiny {}]{v10}
\Vertex[x = 1.64999997615814, y = 0.399999976158142, L = \tiny {}]{v11}
\Vertex[x = 1.45000004768372, y = 0.399999976158142, L = \tiny {}]{v12}
\Vertex[x = 1.14999997615814, y = 0.399999976158142, L = \tiny {}]{v13}
\Vertex[x = 1.29999995231628, y = 0.399999976158142, L = \tiny {}]{v14}
\Vertex[x = 0.649999976158142, y = -0.25, L = \tiny {}]{v15}
\Vertex[x = 0.949999988079071, y = -0.25, L = \tiny {}]{v16}
\Vertex[x = 0.800000011920929, y = 0, L = \tiny {}]{v17}
\Vertex[x = 1.29999995231628, y = 0, L = \tiny {}]{v18}
\Vertex[x = 1.45000004768372, y = -0.25, L = \tiny {}]{v19}
\Vertex[x = 1.14999997615814, y = -0.25, L = \tiny {}]{v20}
\Vertex[x = 1.64999997615814, y = -0.25, L = \tiny {}]{v21}
\Vertex[x = 1.95000004768372, y = -0.25, L = \tiny {}]{v22}
\Vertex[x = 1.79999995231628, y = 0, L = \tiny {}]{v23}
\Edge[](v0)(v3)
\Edge[style = {bend left}](v0)(v5)
\Edge[style = {ultra thick}](v0)(v6)
\Edge[style = {ultra thick}](v0)(v7)
\Edge[style = {ultra thick}](v0)(v8)
\Edge[](v1)(v2)
\Edge[style = {bend right}](v1)(v4)
\Edge[style = {ultra thick}](v1)(v6)
\Edge[style = {ultra thick}](v1)(v7)
\Edge[style = {ultra thick}](v1)(v8)
\Edge[style = {bend right}](v2)(v5)
\Edge[style = {ultra thick}](v2)(v9)
\Edge[style = {ultra thick}](v2)(v10)
\Edge[style = {ultra thick}](v2)(v11)
\Edge[style = {bend left}](v3)(v4)
\Edge[style = {ultra thick}](v3)(v9)
\Edge[style = {ultra thick}](v3)(v10)
\Edge[style = {ultra thick}](v3)(v11)
\Edge[style = {ultra thick}](v4)(v12)
\Edge[style = {ultra thick}](v4)(v13)
\Edge[style = {ultra thick}](v4)(v14)
\Edge[style = {ultra thick}](v5)(v12)
\Edge[style = {ultra thick}](v5)(v13)
\Edge[style = {ultra thick}](v5)(v14)
\Edge[style = {ultra thick}](v6)(v15)
\Edge[](v6)(v19)
\Edge[](v6)(v22)
\Edge[style = {ultra thick}](v7)(v16)
\Edge[](v7)(v20)
\Edge[](v7)(v23)
\Edge[style = {ultra thick, bend right}](v8)(v17)
\Edge[](v8)(v18)
\Edge[](v8)(v21)
\Edge[](v9)(v16)
\Edge[](v9)(v18)
\Edge[style = {ultra thick}](v9)(v22)
\Edge[](v10)(v17)
\Edge[](v10)(v19)
\Edge[style = {ultra thick, bend right}](v10)(v23)
\Edge[](v11)(v15)
\Edge[](v11)(v20)
\Edge[style = {ultra thick}](v11)(v21)
\Edge[](v12)(v16)
\Edge[style = {ultra thick}](v12)(v19)
\Edge[](v12)(v21)
\Edge[](v13)(v17)
\Edge[style = {ultra thick}](v13)(v20)
\Edge[](v13)(v22)
\Edge[](v14)(v15)
\Edge[style = {ultra thick, bend right}](v14)(v18)
\Edge[](v14)(v23)
\Edge[style = {ultra thick}](v15)(v16)
\Edge[style = {ultra thick}](v15)(v17)
\Edge[style = {ultra thick}](v16)(v17)
\Edge[style = {ultra thick}](v18)(v19)
\Edge[style = {ultra thick}](v18)(v20)
\Edge[style = {ultra thick}](v19)(v20)
\Edge[style = {ultra thick}](v21)(v22)
\Edge[style = {ultra thick}](v21)(v23)
\Edge[style = {ultra thick}](v22)(v23)
\end{tikzpicture}

\caption{Graphs with $\Delta=4$ and $\Delta=5$ that have squares
$K_{4^2-1}$ and $K_{5^2-1}$.\label{Elspas-fig}}
\end{figure}

Now we want to show that $G^2\ne K_{\Delta^2+1}$ implies the stronger result
$\chi^2(G)\le \Delta^2-1$. 
The first work in this direction is by Cranston and
Kim~\cite{CranstonK08}; they showed that if $\Delta=3$, $G$ is connected, and
$G$ is not the Petersen graph, then $\chil^2(G)\le\Delta^2-1=8$.  They also
conjectured that $\chil^2(G)\le\Delta^2-1$ for every connected graph $G$ such
that $\Delta\ge 3$ and $G^2\ne K_{\Delta^2+1}$.  Cranston and
Rabern~\cite{CranstonR13paint}
proved this bound in the more general setting of Alon--Tarsi number.

\begin{theorem} [\cite{CranstonR13paint}]
\label{CRpaint}
If $\Delta\ge 3$ and $G$ is a connected graph other than the Petersen graph,
the Hoffman--Singleton graph, or a graph with $\Delta=57$ and $G^2=K_{57^2+1}$,
then $\chiAT^2(G)\le\Delta^2-1$.
\end{theorem}

\begin{proof}[Proof Sketch.]
The idea of the proof is straightforward; for simplicity
we focus on list coloring, although the extension to Alon--Tarsi number follows
the same approach.  Recall the proof of Brooks' Theorem for list
coloring, due to \erdos, Rubin, and Taylor~\cite{ERT}.  They choose some
connected subgraph $H$ and greedily color the vertices in order of decreasing
distance from $H$.  At the time that each vertex $v$ outside of $H$ is colored,
it has some uncolored neighbor closer to $H$, so $v$ can be colored from a list
of size $\Delta$.  Thus, it suffices to pick $H$ that is degree-choosable,
i.e., it can be list colored whenever each vertex $v$ of $H$ has a list of size
$d_H(v)$.  So the main work in their proof consists in showing that each
2-connected graph (that is neither a complete graph nor an odd cycle) contains
such a subgraph $H$.

The proof of Theorem~\ref{CRpaint} follows the same outline.  Again, we choose
some subgraph $H$ of $G$ and color the vertices of $G^2$ in order of decreasing
distance in $G$ from $H$.  Now each vertex $v$ has at least \emph{two}
uncolored neighbors when it is colored, so lists of size $\Delta^2-1$ suffice.
In~\cite{ERT} the choice of $H$ is always an induced even cycle with at
most one chord.  Here we choose $H$ similarly; when $G$ has sufficiently high
girth, $H$ is the square of a shortest cycle in $G$, possibly with one or two
pendant edges.  The case of smaller girth is more detailed.  In
particular, girth 5 requires careful analysis, since it 
explicitly uses that $G^2\ne K_{\Delta^2+1}$.  
\end{proof}

Now we continue the task of proving Wegner's Conjecture for $k=2$ and
specific values of $D$.  The Petersen graph and the Hoffman--Singleton graph,
together with greedy coloring, prove the conjecture for $D\in\{3,7\}$.
For all other $D$, besides possibly 57, Theorem~\ref{CRpaint} gives an upper
bound.  The constructions of Elspas~\cite{Elspas64} for $D=4$ and $D=5$ give
matching lower bounds.  Thus, for $k=2$, Wegner's Conjecture is true for all
$D\in\{3,4,5,7\}$.  The connected graphs achieving the upper upper bound are
unique for $D\in\{3,7\}$, as proved in~\cite{HoffmanS60}; and also for  
$D\in\{4,5\}$, and the graphs shown in Figure~\ref{Elspas-fig}, as proved
in~\cite{NguyenM}.
For easy reference, we present these results in the following theorem.

\begin{theorem}
Wegner's Conjecture is trivial for $k=1$, as shown by the
complete graph $K_{D+1}$.  For $k=2$ it is true for $D\in \{3,4,5,7\}$. 
(As far as we know, it is open for all other pairs $(k,D)$.)
\end{theorem}

We might hope to find similar constructions for larger $D$.  But this seems
unlikely.  Conde and Gimbert~\cite{CondeG09} showed that for each $D$ with $6\le D\le
49$, there exists no graph $G$ with $\Delta=D$ and $G^2=K_{\Delta^2-1}$.  Miller,
Nguyen, and Pineda--Villavicencio~\cite{MillerNP-V09} conjectured this 
for all $D\ge 6$, and proved it for various cases.  Miller and
\v{S}ir\'{a}n~\cite[p.~13]{MillerSDynamic} summarized~these~results.

Relatively little is known about lower bounds for $\omega^2(D)$.
Recall that for each prime power $q$, there exists a projective plane
$\mathcal{P}$ of order $q$.  Brown~\cite{Brown66} considered the bipartite
incidence graph $G$ of $\mathcal{P}$, which has as its two parts the $q^2+q+1$
points and $q^2+q+1$
lines of $\mathcal{P}$.  Since every pair of lines intersect in a common point,
in $G^2$ the lines form a clique of size $q^2+q+1$ (similarly for the points).
Since $G$ is $(q+1)$-regular, $\omega(G^2)=q^2+q+1=\Delta^2-\Delta+1$.  
When $\Delta-1$ is not a power of a prime,
the best bounds known still come from this construction, together with the
fact~\cite{PrimeGaps01} that for $\Delta$ sufficiently large there always
exists a prime $p$ with $\Delta\ge p\ge \Delta-\Delta^{0.525}$.
This lack of better constructions prompts the following question, which arose
from discussion with Goddard.

\begin{question}
\label{question1}
For each integer $t$, does there exist a constant $\Delta_t$ such that
all graphs $G$ with $\Delta\ge \Delta_t$ satisfy $\omega^2(G)\le \Delta^2-t$?
Can this conclusion be strengthened further to 
(i) $\chi^2(G)\le \Delta^2-t$,
(ii) $\chil^2(G)\le \Delta^2-t$,
(iii) $\chip^2(G)\le \Delta^2-t$,
or even 
(iv) $\chiAT^2(G)\le \Delta^2-t$?
\end{question}

%The Borodin--Kostochka Conjecture:
Before we leave this section, we mention the following related conjecture of
Borodin and Kostochka~\cite{BorodinK77}:
For every graph $G$ with $\Delta\ge 9$, if $\omega(G)<\Delta$, then
$\chi(G)\le \Delta-1$.  When restricted to graphs that are squares, this
conjecture implies the coloring analogue of Theorem~\ref{CRpaint}.  In general
the Borodin--Kostochka Conjecture remains open, although
Reed~\cite{Reed99brooks} used the probabilistic method to prove it for $\Delta$
sufficiently large\footnote{He showed that $\Delta\ge 10^{14}$ suffices.
However, he commented in the paper that with more detailed analysis% of the argument
, this could be reduced to $\Delta\ge 10^6$ and maybe even $\Delta\ge
10^3$ (but probably not to $\Delta\ge 10^2$).}.
However, his result is unlikely to help resolve Wegner's conjecture
for any further values of $D$, since quite probably $\omega^2(D)<\Delta^2-1$ for
all $D\ge 6$, with the possible exception of $D=57$.

\section{Planar Graphs and Sparse Graphs}
\label{sec:planar-sparse}
\subsection{Wegner's Planar Graph Conjecture}
\label{sec:wegner}

The most well-known conjecture on coloring squares was made by
Wegner~\cite{Wegner77}, in 1977.

\begin{wegner}[\cite{Wegner77}]
\label{wegner-planar-conj}
If $G$ is planar with maximum degree $\Delta$, then

\begin{align*}
\chi^2(G) \le
\left\{
\begin{array}{ll}
7   & \mbox{if $\Delta=3$} \\
\Delta+5 & \mbox{if $4\le \Delta\le 7$} \\
\Floor{\frac32\Delta}+1 & \mbox{if $\Delta\ge 8$.} 
\end{array}
\right. 
\end{align*}
\end{wegner}

\begin{figure}[!ht]
\centering
\begin{tikzpicture}[thick, scale=.5]
\tikzstyle{uStyle}=[shape = circle, minimum size = 5.5pt, inner sep = 0pt,
outer sep = 0pt, draw, fill=white, semithick]
\tikzstyle{lStyle}=[shape = circle, minimum size = 4.5pt, inner sep = 0pt,
outer sep = 0pt, draw, fill=none, draw=none]
\tikzstyle{usStyle}=[shape = circle, minimum size = 4.5pt, inner sep = 0pt,
outer sep = 0pt, draw, fill=black, semithick]
\tikzset{every node/.style=uStyle}
\def\rad{2.25cm}

\draw (0,0) node (origin) {};
\foreach \ang in {90, 210, 330}
\draw (origin) -- (\ang:\rad*.5) node {} -- (\ang:\rad) node (x\ang) {};

\draw (x90) -- (x210) -- (x330) -- (x90);

\begin{scope}[xshift=2.5in, yshift=.225in, scale=1.07]
\draw (0,0) node (origin) {};
\foreach \ang in {0, 90, 180, 270}
\draw (origin) -- (\ang:\rad*.45) node (x\ang) {} -- (\ang+45:\rad) node (y\ang) {};

\draw (x0) -- (y270) -- (y0) (x90) -- (y0) -- (y90) (x180) -- (y90) -- (y180)
(x270) -- (y180) -- (y270);

\end{scope}

\begin{scope}[xshift=5in]
\draw (0,0) node (origin) {};
\foreach \ang in {90, 210, 330}
\draw (origin) -- (\ang:\rad*.5) node (x\ang) {} -- (\ang:\rad) node (y\ang) {}
(\ang:-.5*\rad) node (z\ang) {};

\draw (x90) -- (x210) -- (x330) -- (x90);
\draw (y90) edge[bend right=25] (y210);
\draw (y210) edge[bend right=25] (y330);
\draw (y330) edge[bend right=25] (y90);

\draw (y90) -- (z330) -- (y210) (z330) -- (x210);
\draw (y210) -- (z90) -- (y330) (z90) -- (x330);
\draw (y330) -- (z210) -- (y90) (z210) -- (x90);
\end{scope}

\begin{scope}[xshift=8.0in, yshift=-.43in, yscale=.555, xscale=1.0]
\draw (0,0) node (A) {};
\draw (0,6) node (B) {};
\draw (-3,4) node (C) {};
\draw (-3,2) node (D) {};
\draw (-2,4) node (E) {};
\draw (-2,2) node (F) {};
\draw (-1,4) node (G) {};
\draw (-1,2) node (H) {};
\draw (0,3) node (J) {};
\draw (1,3) node (K) {};
\draw (2,3) node (L) {};
\draw (3,3) node (M) {};

\draw (A) -- (J) -- (B) -- (C) -- (D) -- (A) -- (F) -- (E) -- (B) -- (G) -- (H)
-- (A) -- (K) -- (B) -- (L) -- (A);
\draw (C) -- (E) -- (D) (G) -- (E) -- (H);
\draw[thick, black, decorate, decoration={snake, amplitude = .3mm, segment length=1.4mm}]
(A) -- (M) -- (B);
\draw (M) node{};
\end{scope}

\end{tikzpicture}
\captionsetup{width=.80\linewidth}
\caption{Diameter 2 planar graphs with maximum degrees 3, 4, 5, 7(6) and orders
7, 9, 10, 12(11), respectively.  (To form the graph of order 11 from that of
order 12, delete the rightmost vertex.) These graphs show that Wegner's
Planar Graph Conjecture is best possible for $\Delta\le 7$.
The case $\Delta\ge 8$ is shown in Figure~\ref{fig:wegner}.%
\label{Wegner-small-pics}%
}
\end{figure}

For all $\Delta$, Wegner gave constructions showing this conjecture is best possible.
Each construction is a diameter 2 graph, that is, a graph $G$ such that $G^2$
is a clique.  For $3\le \Delta\le 7$, the constructions are shown in
%in~\cite{HellS} and~\cite[p.~34]{CranstonW-guide}.  
Figure~\ref{Wegner-small-pics}. (To form the graph with $\Delta=6$ from that
with $\Delta=7$, delete the rightmost vertex.)
For $\Delta\ge 8$,
Figure~\ref{fig:wegner} below shows the construction when $\Delta$ is
even; note that all $3s+1$ vertices of $G_s^2$ form a clique, so
$\chi^2(G_s)=3s+1$.  When $\Delta$ is odd, say $\Delta=2s+1$, we add one more
vertex $v_{s+1}$, adjacent to $v$ and $w$.  Hell and Seyffarth~\cite{HellS}
proved that, for each $\Delta\ge 8$, no planar diameter 2 graph has more
vertices than the one constructed by Wegner.

A natural next step in this direction is to prove that every planar graph $G$ with 
maximum degree $\Delta$ satisfies $\omega(G^2)\le 
\Floor{\frac32\Delta}+1$.
This was accomplished in \cite{cranston-planar-cliques} whenever $\Delta\ge 36$.
Furthermore, if $S$ is a clique in $G^2$ of size close to $\Floor{\frac32\Delta}+1$,
then $S$ ``looks like'' the example in Figure~\ref{fig:wegner}.  More formally, we 
have the following result: If $G$ is a plane graph with $\Delta(G)\ge 19$ and $S$
is a maximal clique in $G^2$ with $|S|\ge \Delta+20$, then there exist $x,y,z\in V(G)$
such that $S=\{w:|N[w]\cap\{x,y,z\}|\ge 2\}$.

\begin{figure}[ht]
\centering
\begin{tikzpicture}[thick]
\tikzstyle{uStyle}=[shape = circle, minimum size = 5.5pt, inner sep = 0pt,
outer sep = 0pt, draw, fill=white, semithick]
\tikzset{every node/.style=uStyle}

\def\iirad{1.15cm}
\def\irad{1.5cm}
\def\orad{3.00cm}
\def\oorad{3.35cm}
% nodes
\foreach \lab/\ang/\dist in {v/150/\irad, v1/90/\irad, vk/90/\orad, w/30/\irad,
w1/330/\irad, wk/330/\orad, x/270/\irad, x1/210/\irad, xk/210/\orad}
\draw (\ang:\dist) node (\lab) {};
% labels
\foreach \lab/\ang/\dist in {v/148/1.20cm, v_1/90/1.25cm, v_s/90/\oorad,
w/32/1.20cm, w_1\,/330/1.20cm, ~~~~w_{s-1}/330/\oorad, x/270/\iirad,
x_1/210/1.20cm, x_{s-1}~~/210/\oorad}
\draw (\ang:\dist) node[draw=none, fill=none] {\footnotesize{$\lab$}};

\draw (v) -- (v1) -- (w) -- (w1) -- (x) -- (x1) -- (v);
\draw (v) -- (vk) -- (w) -- (wk) -- (x) -- (xk) -- (v) -- (x) -- (w);
\draw[dotted] (v1) -- (vk) (w1) -- (wk) (x1) -- (xk);
\end{tikzpicture}

\captionsetup{width=.45\linewidth}
\caption{A planar graph $G_s$ with $\Delta(G_s) = 2s$ and
$\chi^2(G_s)=\omega^2(G_s)=|V(G_s)|=3s+1$.
\label{fig:wegner}
}
\end{figure}

For planar $G$, with $\Delta$ moderately small, 
the current best bound on $\chi^2(G)$,
is due to Molloy and Salavatipour~\cite{MolloyS}.

\begin{theorem} [\cite{MolloyS}]
\label{MolloyS-thm}
If $G$ is a planar graph, then $\chi^2(G)\le \Ceil{\frac53\Delta}+78$.
If also $\Delta\ge 241$, then $\chi^2(G)\le \Ceil{\frac53\Delta}+25$.
\end{theorem}

This result does not extend to $\col^2(G)$ or even to $\chil^2(G)$. 
However, asymptotically we can do much better.
Havet, van den Heuvel, McDiarmid, and Reed~\cite{HavetHMR} proved that Wegner's
Planar Graph Conjecture is true in an asymptotic sense, even for list coloring.  

\begin{theorem}[\cite{HavetHMR}]
\label{HvdHMR-thm}
For each $\epsilon>0$ there exists $\Delta_{\epsilon}$ such
that if $G$ is a planar graph with $\DeltaG\ge \Delta_{\epsilon}$, then 
$\chil^2(G)\le\frac32\DeltaG(1+\epsilon)$.
\end{theorem}

In fact, \cite{HavetHMR} proved Theorem~\ref{HvdHMR-thm} more generally for
every ``nice'' class of graphs.  A class of graphs $\G$ is
\Emph{nice} if $\G$ is minor-closed and $\G$ does not contain the complete
bipartite graph $K_{3,t}$ for some positive integer $t$.  (For example, the
class of all planar graphs is nice, since $K_{3,3}$ is non-planar. More
generally, for every surface $S$, the class of graphs embeddable in $S$ is nice.)
One consequence of Theorem~\ref{HvdHMR-thm} is that every nice class of graphs
satisfies $\omega^2(G)\le\frac32\Delta(G)(1+o(1))$ as $\Delta(G)\to \infty$.
For the class of graphs embeddable in each fixed surface $S$,  
Amini, Esperet, and van den Heuvel~\cite{AEvdH}, proved the stronger bound
$\omega^2(G)\le\frac32\Delta(G)+C_S$, for some constant $C_S$ (dependent on $S$).

The proof of Theorem~\ref{HvdHMR-thm}
is long and technical.  But in a sense, the key idea is that list coloring
squares of planar graphs is very similar to list coloring line graphs.  This
problem was solved asymptotically by Kahn~\cite{Kahn2}, and the proof of the key
lemma in~\cite{HavetHMR} follows the approach in~\cite{Kahn2}.
In~\cite{HavetHMR} 
the authors also posed the following conjecture.
It is a special case of the Square List Coloring Conjecture, which we discuss in
Section~\ref{sec:LCS} and which was previously disproved.

\begin{conjecture}[\cite{HavetHMR}]
\label{HvdHMR-conj}
If $G$ is a planar graph, then $\chil^2(G)=\chi^2(G)$.
\end{conjecture}

It is difficult to prove general lower bounds on $\chi^2(G)$ that are stronger
than $\Delta+1$.  So most of the graphs for which Conjecture~\ref{HvdHMR-conj}
has been proved are those for which $\chil^2(G)=\Delta+1$.  We discuss these
in more detail in Section~\ref{sec:planar-girth}.  However, in 2022 this
conjecture was disproved by Hasanvand~\cite{hasanvand}.
%(The conjecture does remain open for graphs with maximum degree at least 4.)
(In the same paper, Hasanvand also constructed various other examples of graphs
$G$ with $\chil^2(G)>\chi^2(G)$.  One such family~\cite[Theorem~3.5]{hasanvand} 
consists of certain $4$-regular planar graphs;
but each of these graphs is also a line graph of a 3-regular (planar) graph, so we defer further
details to our discussion of strong edge-coloring in Section~\ref{strong-subcubic-sec}.)

\begin{theorem}[\cite{hasanvand}]
\label{HvdHMR-c/e}
Conjecture~\ref{HvdHMR-conj} is false.  Specifically, there exists a cubic
claw-free planar graph $G$ of order 12 (see Figure~\ref{fig:hasanvand}) such
that $\chi^2(G)=4<\chil^2(G)$.  In fact, this graph is the first in an infinite
family of (2-connected) cubic claw-free planar graphs satisfying the same inequality.
\end{theorem}
The proof is pretty, so we sketch it.
\begin{proof}[Proof Sketch.]
Form $G$ from $K_4$ by subdiving each edge once and taking the line graph; see
Figure~\ref{fig:hasanvand}.  
Clearly $\chi^2(G)\ge \omega^2(G)=4$.
And the coloring on the left shows that $\chi^2(G)\le 4$. 

\begin{figure}[!b]
\centering
\begin{tikzpicture}[thick, scale=.65]
\tikzstyle{uStyle}=[shape = circle, minimum size = 5.5pt, inner sep = 0pt,
outer sep = 0pt, draw, fill=white, semithick]
\tikzstyle{lStyle}=[shape = circle, minimum size = 4.5pt, inner sep = 0pt,
outer sep = 0pt, draw, fill=none, draw=none]
\tikzstyle{usStyle}=[shape = circle, minimum size = 4.5pt, inner sep = 0pt,
outer sep = 0pt, draw, fill=black, semithick]
\tikzset{every node/.style=uStyle}
\def\radA{1.65cm}
\def\radB{.45cm}
\def\radC{.53cm}

\foreach \name/\ang/\labA/\labB/\labC/\labD in {x/90/3/4/1/2, y/210/1/4/2/3,
z/330/2/4/3/1}
{
\draw (\ang:\radA*.6) node (\name1) {} --++ (\ang:\radA) node (\name2) {} --++
(\ang+70:\radA) node (\name3) {} -- (\name2) --++ (\ang-70:\radA) node (\name4) {} -- (\name3);

\draw (\name1) ++ (\ang-60:\radB) node[lStyle] {$\labA$};
\draw (\name2) ++ (\ang-110:\radB) node[lStyle] {$\labB$};
\draw (\name3) ++ (\ang+70:\radB) node[lStyle] {$\labC$};
\draw (\name4) ++ (\ang-70:\radB) node[lStyle] {$\labD$};
}

\draw (x1) -- (y1) -- (z1) -- (x1);
\draw (x3) -- (y4) (y3) -- (z4) (z3) -- (x4);

\begin{scope}[xshift=4in]
\foreach \name/\ang/\labA/\labB/\labC/\labD in
{x/90/{\overline{3}}/{\overline{3}}/{\overline{2}}/{\overline{1}}, 
y/210/{\overline{1}}/{\overline{1}}/{\overline{3}}/{\overline{2}}, 
z/330/{\overline{2}}/{\overline{2}}/{\overline{1}}/{\overline{3}} 
}
{
\draw (\ang:\radA*.6) node (\name1) {} --++ (\ang:\radA) node (\name2) {} --++
(\ang+70:\radA) node (\name3) {} -- (\name2) --++ (\ang-70:\radA) node (\name4) {} -- (\name3);

\draw (\name1) ++ (\ang-60:\radC) node[lStyle] {$\labA$};
\draw (\name2) ++ (\ang-110:\radC) node[lStyle] {$\labB$};
\draw (\name3) ++ (\ang+70:\radC) node[lStyle] {$\labC$};
\draw (\name4) ++ (\ang-70:\radC) node[lStyle] {$\labD$};
}

\draw (x1) -- (y1) -- (z1) -- (x1);
\draw (x3) -- (y4) (y3) -- (z4) (z3) -- (x4);

\end{scope}
\end{tikzpicture}
\caption{Left: A 4-coloring of $G^2$.  Right: A 4-assignment $L$ such that $G^2$
admits no $L$-coloring. This example disproves Conjecture~\ref{HvdHMR-conj}; see
Theorem~\ref{HvdHMR-c/e}.\label{fig:hasanvand}}
\end{figure}

To prove that $\chil^2(G)> 4$, let $\overline{i}:=\{1,2,3,4,5\}\setminus\{i\}$
for each $i\in \{1,2,3,4,5\}$.  Let $L$ be the 4-assignment shown on the right
in Figure~\ref{fig:hasanvand}.  It is straightforward to check that the only (four)
independent sets of size 3 in $G^2$ are the color classes in the coloring on the
left of Figure~\ref{fig:hasanvand}.  Since $|\cup_{v\in V(G)}L(v)|=5$, if $G^2$
admits an $L$-coloring $\vph$, then two of these independent sets of size 3
must each receive a common color.  

It is easy to check that in
each independent set of size 3 some vertex is missing color $i$, for each
$i\in\{1,2,3\}$.  So $\vph$ must use color 4 on one independent set of size 3,
call it $I_1$, and must use color 5 on another independent set of size 3, call
it $I_2$.  Up to symmetry, we have only two choices for $I_1$ and $I_2$: either
(a) one of these is color class 4 on the left or (b) both of these are among
color classes 1, 2, and 3.  In each case, $G^2[V(G)\setminus (I_1\cup I_2)]$
is a 6-cycle with one chord (forming two 4-cycles).  Further, the resulting
2-assignment for this graph is isomorphic to $\{1,2\}, \{1,3\}, \{2,3\},
\{1,2\}, \{1,3\}, \{2,3\}$ (around the 6-cycle) with the ends of the chord
having the same list.  It is easy to check that for these lists $L'$ the
remaining graph admits no $L'$-coloring.

Now we construct infinitely many such graphs.
Form $K_4'$ from the complete graph $K_4$ by deleting an edge $vw$ and adding
pendent edges $vv'$ and $ww'$.  To transform the above graph $G$ to a larger
graph $G'$ satisfying the inequality $\chi^2(G')=4<\chil^2(G')$, we delete some
edge $xy$ that is not in a 3-cycle and identify $x$
and $y$, respectively, with vertices $v'$ and $w'$ in $K_4'$.  It is easy to
check that a 4-coloring of $G^2$ extends to a 4-coloring of $(G')^2$ (in exactly
two ways).  We extend the 4-assignment $L$ for $G^2$ to a 4-assignment $L'$ for
$(G')^2$ by giving each new vertex the list that is already assigned to both
$x$ and $y$.  It is straightforward to check that any $L'$-coloring of $(G')^2$
restricts to an $L$-coloring of $G^2$, which cannot exist.  Thus, $(G')^2$ has
no $L'$-coloring.  Finally, we can apply this construction repeatedly to build
the above-mentioned infinite family of counterexamples to
Conjecture~\ref{HvdHMR-conj}.
\end{proof}

Before returning to Wegner's Planar Graph Conjecture, we mention one other
intriguing conjecture from~\cite{HavetHMR}. 

\begin{conjecture}[\cite{HavetHMR}]
\label{wegner-generalization-conj}
If $\G$ is a nice class of graphs, then there exists $\Delta_0$ such that
if $G\in \G$ and $\Delta(G)\ge \Delta_0$, then $\chi^2(G)\le \chil^2(G)\le
\left\lfloor\frac32\Delta(G)\right\rfloor+1$.
\end{conjecture}
This conjecture is best possible, since equality holds for the graphs in
Figure~\ref{fig:wegner}. 

Although Wegner's Planar Graph Conjecture has been solved asymptotically,
whenever $\Delta\ge 4$ the conjecture remains open.
In contrast, the case $\Delta=3$ has been confirmed by
Thomassen~\cite{thomassen-wegner3} and independently by Hartke, Jahanbekam, and
Thomas~\cite{HJT}.  

\begin{theorem}[\cite{thomassen-wegner3,HJT}]
\label{wegner3-thm}
If $G$ is planar and $\DeltaG\le 3$, then $\chi^2(G)\le 7$.
\end{theorem}

The two
proofs of Theorem~\ref{wegner3-thm} take strikingly different approaches. 
The work of Hartke et al.~uses a fairly straightforward discharging argument,
together with extensive computer case-checking to
prove that various configurations are reducible.  The proof of Thomassen relies
on a detailed structural analysis, and also uses the Four Color Theorem.
\bigskip

Many of the bounds known on $\chi^2(G)$ and $\chil^2(G)$ for planar graphs
 actually hold for $\col^2(G)$, although the authors often don't state this fact.
To illustrate the idea of such a proof, we show that every planar graph $G$ satisfies
$\col^2(G)\le9\DeltaG-3$.  Let $v_1,\ldots,v_n$ be an order of the vertices
showing that $\col(G)\le 6$.  More precisely, let $v_1,\ldots,v_n$ be a vertex
order and let $G_i:=G\setminus\{v_1, v_2, \ldots, v_{i-1}\}$ such that
$d_{G_i}(v_i)\le 5$ for all $i$.  Similarly, let $G_i^2:=G^2\setminus\{v_1, v_2,
\ldots,v_{i-1}\}$.  We will show, for each $i$, that $d_{G^2_i}(v_i)\le 9\DeltaG$.

Recall that $v_i$ has at most five neighbors in $G_i$; each of these neighbors has
at most $\DeltaG-1$ other neighbors in $G_i$, for a total of at most $5\DeltaG$
neighbors in $G^2_i$.  In addition, $v_i$ may be adjacent in $G_i^2$ to some
vertex $w$ such that $x$ is a common neighbor of $v_i$ and $w$ in $G$, but $x$
precedes $v_i$ in the order.  Each such $x$ is followed by at most 4 neighbors in
the order, other than $v_i$.  Since $v_i$ is preceded by at most $\Delta$ such
vertices $x$, it has at most $4\Delta$ neighbors in $G_i^2$ of this type.
Thus, $d_{G^2_i}(v)\le 5\DeltaG+4\DeltaG=9\DeltaG$.  
Hence $\col^2(G)\le 9\Delta+1$. 
(More generally~\cite{KrumkeMR01}, if $G$ is $k$-degenerate, then $G^2$ is
$((2k-1)\Delta)$-degenerate.)

The first result in this direction was due to T.~K.~Jonas~\cite{Jonas93}.  Work in his
dissertation showed, implicitly, that $\col^2(G)\le 8\Delta-22$ when $\Delta\ge 5$.
(In fact, Jonas studied $L(2,1)$-labelings, which we consider in Section~\ref{sec:L21}.)
Wong~\cite{Wong}, in his masters thesis, strengthened this to $\col^2(G)\le
3\Delta+5$.  Next, van den Heuvel and McGuinness~\cite{vdHM03} showed that
$\col^2(G)\le2\Delta+25$; see also~\cite[Theorem 6.10]{CranstonW-guide} for a simplified
proof that $\col^2(G)\le 2\Delta+34$.  
%\footnote{

More than 4 decades since the work of Jonas, proving bounds on $\chi^2(G)$ for planar graphs $G$ remains 
a popular line of research~\cite{Aoki6,Aoki5,CMZ,Hajjar,HJMZ,KRT,ZHL}, primarily when $\Delta(G)$ is 
fairly small.  Bousquet, Deschamps, de Meyer, and Pierron~\cite{BDMP} proved that if
$G$ is planar with $\Delta\ge 9$, then $\chi^2(G)\le 2\Delta+7$.  This is the
best known bound on $\chi^2(G)$ when $9\le \Delta\le 31$ (however, the proof at one point
requires permuting colors, so it does not extend to list coloring or degeneracy). %}
This result was extended~\cite{Deniz6-8} by Deniz to the case $6\le \Delta\le 8$.
When $\Delta=5$, the best bound~\cite{Deniz5} is 16;
when $\Delta=4$, the best bound~\cite{BDMP-automatic} is 12;
and when $\Delta=3$, as mentioned above, the best bound~\cite{HJT,thomassen-wegner3} is 7, which is sharp.

\captionsetup{width=.725\linewidth}
\begin{figure}[!b]
\centering
\begin{tikzpicture}[thick] %, scale=.75]
\tikzstyle{uStyle}=[shape = circle, minimum size = 6.5pt, inner sep = 0pt,
outer sep = 0pt, draw, fill=white, semithick]
\tikzset{every node/.style=uStyle}

\begin{scope}
\draw (0,0) node (B) {};

\foreach \name/\angle in {A/0, C/60, D/150, E/210, F/300}
	\draw (\angle:1.618cm) node (\name) {};

\draw (A) ++ (-30:1.618cm) node (G) {};
\draw (A) ++ (30:1.618cm) node (H) {};

	\draw (A) -- (C) -- (D) -- (B) -- (E) -- (F) -- (B) -- (C) -- (H) -- (A) -- (G) -- (F) -- (A);
	\draw[line width=.08cm] (A) -- (B) (D) -- (E) (G) -- (H);
\end{scope}

\begin{scope}[xshift=3in]
\draw (0,0) node (B) {};

\foreach \name/\angle in {A/0, C/60, D/150, E/210, F/300}
	\draw (\angle:1.618cm) node (\name) {};
\draw (A) ++ (-30:1.618cm) node (G) {};
\draw (A) ++ (30:1.618cm) node (H) {};

\foreach \to/\from in {A/C, C/D, D/B, B/E, E/F, F/B, B/C, C/H, H/A, A/G, G/F, F/A}
{
	\draw (\to) to node[pos=.5, minimum size=4pt] {} (\from);
	\draw (\to) to [bend left = 20] node[midway, minimum size=4pt] {} (\from);
	\draw (\to) to [bend left = -20] node[pos=.5, minimum size=4pt] {} (\from);
}

	%\draw (A) -- (C) -- (D) -- (B) -- (E) -- (F) -- (B) -- (C) -- (H) -- (A) -- (G) -- (F) -- (A);
	\draw[line width=.08cm] (A) -- (B) (D) -- (E) (G) -- (H);
\foreach \to/\from in {A/B, D/E, G/H}
	\draw (\to) to [bend right = 25] node[midway, minimum size=4pt] {} (\from);
\end{scope}

\end{tikzpicture}
	\caption{Left: A portion of the icosahedron and perfect matching $M$ (shown in bold).
	Right: The corresponding portion of the construction when $k=3$.\label{icos-fig}}
\end{figure}

Earlier, the best possible bound on $\col^2(G)$ was proved, for sufficiently large maximum
degree, by Agnarsson and
Halld\'orsson~\cite{AgnarssonH03} and Borodin, Broersma, Glebov, and van den
Heuvel~\cite{BBGvdH1,BBGvdH2}. 
\begin{theorem}[\cite{AgnarssonH03,BBGvdH1,BBGvdH2}]
If $G$ is planar with $\Delta$ large enough, then  
$\col^2(G)\le\Ceil{\frac95\Delta}+1$.
\end{theorem}
The first group %~\cite{AgnarssonH03} 
proved that $\Delta\ge 750$ suffices %is big enough 
and the second
%~\cite{BBGvdH1,BBGvdH2} 
that $\Delta\ge47$ suffices.  Each group 
%that $\col^2(G)\le 1+\Ceil{\frac95\Delta}$; they 
also gave the same construction (below) showing that this bound is best possible. 

\begin{proposition}[\cite{AgnarssonH03,BBGvdH2}]
For each $\Delta\ge 9$, there exists a planar graph $G_{\Delta}$, with maximum
degree $\Delta$, such that $G_{\Delta}^2$ has minimum degree
$\ceil{\frac95\Delta}$.
\label{degen-prop}
\end{proposition}
\begin{proof}[Proof Sketch.] We just prove the case $\Delta=5k-1$, but the other cases are
similar.  Let $H$ be an icosahedron, embedded in the plane; see the left of Figure~\ref{icos-fig}.  
Note that we can partition the edges of $H$ into five perfect matchings; denote one of these by
$M$.  For each $vw\in E(H)\setminus M$, replace $vw$ by $k$ paths of length 2 joining $v$
and $w$.  For each $vw\in M$, keep $vw$ and add $k-2$ paths of length 2 joining
$v$ and $w$.  Call the resulting graph $G_\Delta$.  

It is easy to check that each $v\in V(G_\Delta)$ has either (a)
$d_{G_\Delta}(v)=5k-1$ or (b) $d_{G_\Delta}(v)=2$.  (a) If $v\in V(G_\Delta)$
and $d_{G_\Delta}(v)=5k-1$, then $v$ is
a vertex of $H$ and $d_{G_\Delta^2}(v) = k-2 + 8k + 5 = 9k+3$.
(b) Suppose instead $v\in V(G_\Delta)$ and $d_{G_\Delta}(v)=2$; so $v$ arose from
some edge $e$ of $G$.  If $e\in E(H)-M$,  
then $d_{G_\Delta^2}(v)=k-1+2+2(3k+k-2+1)=9k-1$.
If instead $e\in M$,  then $d_{G_\Delta^2}(v)=k-3+2+2(4k)=9k-1$.  Thus
$G_\Delta^2$ has minimum
degree $9k-1=\ceil{\frac95\Delta}$, as desired.  When $\Delta\ne 5k-1$, we
treat more of the five perfect matchings as we treated $M$.
\end{proof}

%Recall that 
Theorems~\ref{MolloyS-thm} and~\ref{HvdHMR-thm} both circumvent the
``barrier'' implied
by Proposition~\ref{degen-prop}.  To accomplish this, they must rely on coloring
techniques more sophisticated than simply coloring greedily in the reverse of a
vertex order witnessing the degeneracy of the graph.

We will soon consider coloring squares of planar graphs with high girth.
To end this section, we offer a brief motivation.
It is natural to search for classes of graphs where we can prove an upper bound
on $\chi^2$ closer to the trivial lower bound of $\Delta+1$.  For the class of
all planar graphs, the best upper bound we can hope for is
$\chi^2\le\left\lfloor\frac32\Delta\right\rfloor+1$, as witnessed by the graphs 
in Figure~\ref{fig:wegner}.  To prove a stronger upper bound, we must exclude
these graphs.  One natural approach is to forbid cycles of certain lengths.
These examples contain cycles only with lengths in $\{3,4,5,6\}$.  But the
3-cycles and 5-cycles are not crucial for the construction.
If we subdivide edges $vx$ and $wx$, the resulting graph $G'$ is bipartite
with $\chi^2(G')=3s$.  Thus, to significantly improve the bound, we
must exclude either 4-cycles or 6-cycles.  One obvious way to do the former
is to require girth at least 5.  This motivates the topic of the next section.
%
%Because of the examples in Figure~\ref{fig:wegner}, we cannot prove any bound on
%$\chi^2(G)$ less than $\Floor{\frac32\Delta}+1$.  If we seek such a bound, then we
%must restrict the class of graphs to forbid those in Figure~\ref{fig:wegner}. 
%One natural choice is to consider graphs with higher girth.  Although
%Figure~\ref{fig:wegner} does contain two 3-cycles, these are not crucial for
%the construction.  If we subdivide edges $vx$ and $wx$, the resulting graph $G'$
%has girth 4 and $\chi^2(G')=3s$.  Thus, to significantly improve the bound, we
%require girth at least 5.  This is the topic of the next section.

\subsection{Planar Graphs of High Girth}
\label{sec:planar-girth}

Wang and Lih~\cite{WangL03} were among the first to study $\chi^2$ for planar
graphs with high girth.  (In fact, for these graphs they considered
$L(2,1)$-labelings, which we discuss briefly in Section~\ref{sec:L21}.)  If $G$
has girth at least 5, then clearly the construction in Figure~\ref{fig:wegner}
is excluded.  This absence of constructions led Wang and Lih to ask whether for
$\Delta$ sufficiently large the trivial lower bound, $\Delta+1$, is also an
upper bound.

\begin{conjecture}[\cite{WangL03}]
For every $k\ge 5$ there exists $\Delta_k$ such that if $G$ is a planar graph
with girth at least $k$ and $\Delta\ge \Delta_k$ then $\chi^2(G)=\Delta+1$.
\label{conj:wang-lih}
\end{conjecture}

Borodin et~al.~\cite{BorodinGINT04} proved the Wang--Lih Conjecture for $k\ge
7$.  
\begin{theorem}[\cite{BorodinGINT04}]
If $G$ is planar with girth at least 7 and $\Delta\ge 30$, then 
$\chi^2(G)=\Delta+1$.
\end{theorem}
In contrast, for each integer $D$ at
least 3, they constructed~\cite{BorodinGINT04} a planar graph $G_D$ with girth 6 and
$\Delta=D$ such that $\chi^2(G_D)= \Delta+2$; this disproved
Conjecture~\ref{conj:wang-lih} for each $k\in \{5,6\}$.  
Subsequently, Dvo\v{r}\'{a}k et~al.~\cite{DvorakKNS08} gave a slightly simpler
construction, which we present below.

\newcommand\gadget[1]{%
\begin{scope}[xshift=1.75*#1mm, scale=.65]
\draw [thick] (0,0) node (x#1) {} 
  -- (-2,-1) node {} -- (-2,-2) node {} -- (0,-3) node (y#1) {} -- (-1,-2) node {}
  -- (-1,-1) node {} -- (x#1) -- (2,-1) node {} -- (2,-2) node {} -- (y#1) -- (1,-2)
  node {} -- (1,-1) node {} -- (x#1) (y#1) -- (0,-4) node (yz#1) {};
\draw (0,-1.5) node[draw=none] {$\dots$};
\end{scope}%
}

%\iffalse
\begin{figure}[ht]
\centering
\begin{tikzpicture}[scale=.9]
\tikzstyle{uStyle}=[shape = circle, minimum size = 5.5pt, inner sep = 0pt,
outer sep = 0pt, draw, fill=white, semithick]
\tikzset{every node/.style=uStyle}
\foreach \i in {5,30,48,72,90}
{
	\gadget{\i}
}

\draw (10.5,-3.5) node (z) {};
\draw (10.5, .7) node (v) {};
\draw[thick] (10.5, -1.95) node (w) {} -- (z) (w) -- (v);
\draw (10.75,0.37) node[draw=none] {\footnotesize{$v$}};
\draw (10.85,-1.95) node[draw=none] {\footnotesize{$w$}};
\draw (10.75,-3.20) node[draw=none] {\footnotesize{$z$}};
\draw (1.2,-2.05) node[draw=none] {\footnotesize{$y$}};
\draw[thick] (yz5) --++ (0,-.65) node (yp) {};
\draw (yp) ++ (3.25mm,0mm) node[draw=none, shape=rectangle] {\footnotesize{$z$}};
\draw (yp) ++ (1mm,-6.0mm) node[draw=none, shape=rectangle] {\footnotesize{$G'_D$}};
\draw (z) ++ (1mm,-6.0mm) node[draw=none, shape=rectangle] {\footnotesize{$G_D$}};
\draw (9.50,-2.35) node[draw=none] {$\ldots$};
\draw (11.50,-2.35) node[draw=none] {$\ldots$};

\foreach \i in {30,48,72,90}
\draw[thick] (z) -- (yz\i) (x\i) -- (v);

\foreach \i/\lab/\x/\y in {5/~/1/3.2, 30/1/-4.25/1, 48/2/-4.25/1, 72/D-2/6.5/1, 90/D-1/6.5/1}
\draw (x\i) ++ (\x mm, \y mm) node[draw=none, shape=rectangle] {\footnotesize{$x_{\lab}$}};
\end{tikzpicture}

\caption{In any $(D + 1)$-coloring of the square of $G'_D$, the
$(D - 1)$-vertex $x$ and the 1-vertex $z$ must receive distinct colors.
Thus, $\chi(G_D^2) \ge D + 2$.}
\label{lowerbounds}
\end{figure}
%\fi

\begin{proposition}[\cite{BorodinGINT04,DvorakKNS08}]
For every $D\ge 3$, there exists a planar graph $G_D$ with girth 6 and
$\Delta=D$ such that $\chi^2(G_D)=D+2$; see the right of
Figure~\ref{lowerbounds}.  
\label{prop:girth6}
\end{proposition}
\begin{proof}
Consider the graph $G'_D$ on the left in Figure~\ref{lowerbounds}, where
$d(y)=D$.  We show that in any $(D+1)$-coloring of $(G'_D)^2$
vertices $x$ and $z$ must get distinct colors.  By symmetry, say that $y$ is
colored 1, its common neighbor with $z$ is colored $D+1$, and its other
$D-1$ neighbors each get a distinct color from $\{2,\ldots,D\}$.  Now
$x$ is colored from $\{1,D+1\}$, but $z$ is not.  Thus, $x$ and $z$ get
distinct colors.  To see that $\chi^2(G_D)\ge D+2$, let
$S:=\{x_1,\ldots,x_{D-1},v,w,z\}$ on the right in Figure~\ref{lowerbounds}. 
Note that $S\setminus\{z\}$ induces $K_{D+1}$ in $G^2_D$.  Further, $z$ must
get a color distinct from each vertex of $S\setminus\{z\}$.  Thus, any proper
coloring of $G^2_D$ uses at least $D+2$ colors.
\end{proof}

These results lead to two natural questions. (i) For each $k\ge7$, what is the
smallest $D$ such that every planar $G$ with girth $g\ge k$ and $\Delta\ge D$
satisfies $\chi^2(G)=\Delta+1$?  Similarly for $\chil^2(G)$.  (ii) For each $g\in
\{5,6\}$, what is the smallest constant $C$ such that every planar $G$ with
sufficiently large $\Delta$ satisfies $\chi^2(G)\le \Delta+C$?  We address these %two 
questions in order.

Borodin, Ivanova, and Neustroeva~\cite{BIN04}
proved that $\chi^2(G)=\Delta+1$ when $G$ is planar, $\Delta\ge D$, and
$g\ge k$, for all $(D,k)\in\{(3,24), (4,15), (5,13), (6,12), (7,11), (9,10), (16,9)\}$.
(Later, La and Montassier~\cite{LM-potential} proved the same upper bound when $(D,k)=(7,9)$.
And La, Montassier, Pinlou, and Valicov~\cite{LMPV} proved this bound when $(D,k)=(9,8)$.)
A few years later, the authors of~\cite{BIN04} extended their results to list coloring.
They showed~\cite{BIN07} that $\chil^2(G)=\Delta(G)+1$ when
$(D,k)\in\{(3,24), (4,15), (5,13),$ $(6,12), (7,11),$ $(9,10)$, $(15,8)$, $(30,7)\}$.
Shortly thereafter, Ivanova~\cite{Ivanova10} strengthened the bounds for some values
of $k$: If 
$(D,k)\in\{(5,12), (6,10),$ $(10, 8), (16, 7)\}$,  then $\chil^2(G)=\Delta(G)+1$.

\begin{figure}[!h]
\centering
\begin{tabular}{r|cccccccc}
$\Delta\ge$ & 3 & 4 & 5 & 6 & 10 & 16 & $-$\\
\hline
girth $\ge$ & 24 & 15 & 12 & 10 & 8 & 7 & 6
\end{tabular}
\caption{Pairs $(D,k)$, in columns, such that $\chil^2(G)=\Delta+1$ if $G$ is a planar
graph with maximum degree $\Delta\ge D$ and girth $g\ge k$.  (The entry $(-,6)$
denotes that no such pair $(D,6)$ exists. References are in the preceeding few
paragraphs.)}
\end{figure}

In 2008, Dvo\v{r}\'ak, \kral, Nejedl\'{y}, and
\v{S}krekovski~\cite{DvorakKNS08} showed that for $k=6$, the Wang--Lih
Conjecture fails only by 1.

\begin{theorem}[\cite{DvorakKNS08}]
If $G$ is planar with girth at least 6 and
$\Delta\ge8821$, then $\chi^2(G)\le \Delta+2$.
\end{theorem}

By using the same core idea as~\cite{DvorakKNS08}, but a more refined
discharging argument, Borodin and Ivanova strengthened this result: in 2009
they showed~\cite{BorodinI09girth6} that $\Delta\ge 18$ implies
$\chi^2(G)\le\Delta+2$ and also~\cite{BorodinI09girth6choice} that $\Delta\ge
36$ (and later~\cite{BorodinI09girth6choiceB} that $\Delta\ge 24$) implies
$\chil^2(G)\le \Delta+2$.  
For a simpler proof when $\Delta\ge 295$, see~\cite[Section 1.4]{GCM}.
Also, see Theorem~\ref{BLP3-thm}, below.
Furthermore, Dvo\v{r}\'{a}k et~al.~\cite{DvorakKNS08} conjectured that
a similar result holds for girth 5.

\begin{conjecture}[\cite{DvorakKNS08}]
There exists $\Delta_0$ such that if $G$ is a planar graph
with girth at least 5 and $\Delta\ge \Delta_0$ then $\chi^2(G)=\Delta+2$.
\label{conj:girth5}
\end{conjecture}
Conjecture~\ref{conj:girth5} 
was eventually confirmed by  Bonamy, Cranston, and Postle~\cite{BonamyCP}.
\begin{theorem}[\cite{BonamyCP}]
Conjecture~\ref{conj:girth5} is true.
\end{theorem}
Specifically, Bonamy et al.~showed that if $G$ is
a planar graph with girth at least 5 and $\Delta\ge 2,689,601$, then
$\chiAT^2(G)\le \Delta+2$.  The proof was surprising.  While it did rely on
reducible configurations, it did not use discharging (only that planar graphs
have average degree less than 6).  In fact, the key idea was a class of
arbitrarily large reducible configurations called ``regions'' that essentially
consist of two high degree vertices, $v$ and $w$, and all of their low degree
neighbors that lie on $v,w$-paths of length at most 3.

We should mention one other variant of this problem.  To avoid the
construction in Figure~\ref{fig:wegner} it is essential to forbid 4-cycles, but
it is not really necessary to forbid 3-cycles.
Zhu, Lu, Wang, and Chen~\cite{ZhuLWC} were the first to consider $\chi^2(G)$,
when $G$ is a planar graph with no 4-cycle and no 5-cycle.  They proved
$\chi^2(G)\le \Delta+5$ when $\Delta\ge 9$.  
Cranston and Jaeger~\cite{CranstonJ} strengthened this bound, showing that
$\col^2(G)\le \Delta+3$, when $\Delta\ge 32$. 
Zhu, Gu, Sheng, and L\"{u}~\cite{ZGSL} showed further that
$\chil^2(G)\le\Delta+3$ when $\Delta\ge 26$.  Finally, Dong and Xu~\cite{DX}
showed that $\chi^2(G)\le \Delta+2$ when $\Delta\ge 185,760$.

All the results above forbid both 4-cycles and 5-cycles, but we can
also think about just forbidding 4-cycles.  (Recall that Figure~\ref{fig:wegner}
shows that we must forbid 4-cycles.)
Wang and Cai~\cite{WC} appear to have studied this problem first.  They proved
that if $G$ is a planar graph with no 4-cycle, then $\chi^2(G)\le \Delta+48$.
Choi, Cranston, and Pierron~\cite{CCP} considered the analogous problem for
the coloring number and sharpened the bound when $\Delta$ is large.
\begin{theorem}[\cite{CCP}]
If $G$ is a planar graph with no 4-cycle, then $\col^2(G)\le \Delta+73$.
If also $\Delta$ is sufficiently large\footnote{\mbox{This proof
shows $\Delta\ge 6\cdot\! 10^{14}$ suffices; but its authors made little
effort optimizing the bound on $\Delta$.}}, then $\chiAT^2(G)\le
\Delta+2$.  
\end{theorem}
The same authors~\cite{CCP} considered the more general problem of coloring the
square of a planar graph $G$ with no cycle having length in some finite set $S$.
They showed that there exists a constant $C_S$ such that always $\chi^2(G)\le
\Delta+C_S$ if and only if $4\in S$.

%\bigskip

Since Conjectures~\ref{conj:wang-lih} and~\ref{conj:girth5} are now completely
resolved, it is natural to extend this line of research as follows.  For every
maximum degree $\Delta$ and girth $g$, what is the minimum constant
$C_{\Delta,g}$ such that $\chi^2(G)\le \Delta+C_{\Delta,g}$ for each planar
graph $G$ with maximum degree $\Delta$ and girth at least $g$?  The same
question can be asked for $\chil^2$.  For this problem, the case $\Delta=3$ has
received more attention than any other; see Figure~\ref{planar-girth-fig}.

Let $G$ be a connected graph, possibly non-planar, with $\Delta=3$. 
Wegner showed that if $G$ is not the Petersen graph, then $\chi^2(G)\le 8$. 
Subsequently, this was strengthened~\cite{CranstonK08, CranstonR13paint} to
$\chil^2(G)\le 8$ and eventually $\chiAT^2(G)\le 8$.  
For non-planar $G$, this is sharp, as witnessed by the Wagner graph, which is
formed from the 8-cycle by adding an edge joining each pair of vertices at
distance 4 on the cycle.  Thus, we seek additional hypotheses on $G$ that
imply better upper bounds on $\chi^2(G)$ and $\chil^2(G)$.  Recall that Wegner
conjectured that $\chi^2(G)\le 7$ for every planar subcubic graph.  
As noted above, this conjecture was proved by Thomassen~\cite{thomassen-wegner3}
and independently by Hartke, Jahanbekam, and Thomas~\cite{HJT}.
A natural next step is planar graphs of higher girth.  

\begin{figure}[!h]
\centering
\begin{tabular}{r|cccccccc}
$\chil^2\le$ & 8 & 7 & 6 & 5 & 4 & 3 \\
\hline
girth $\ge$ & 3 & 6 & 9 & 13 & 24 & $-$
\end{tabular}
\caption{Pairs $(k,g_k)$, in columns, such that $\chil^2(G)\le k$ if $G$ is a planar
graph with maximum degree $\Delta=3$ and girth at least $g_k$.  (The entry $(3,-)$
denotes that no such pair $(3,g_3)$ exists. References are in the following 
paragraph.)\label{planar-girth-fig}}
\end{figure}

Let $G$ be a planar graph with $\Delta=3$ and girth $g$.
Cranston and Kim~\cite{CranstonK08} showed that $\chil^2(G)\le7$ if $g\ge 7$. 
More recently, Kim and Lian~\cite{KL2} showed that $\chil^2(G)\le 7$ if $g\ge 6$.
This same bound was proved by Jin, Kang, and Kim~\cite{JKK} for graphs with no
4-cycles and no 5-cycles; and by Kim, Lian, Nakamoto, and Ozeki~\cite{KLNO} for
graphs with no 5-cycles.
Dvo\v{r}\'{a}k, \v{S}krekovski, and Tancer~\cite{DvorakST08} showed that 
(i) $\chil^2(G)\le 6$ if $g\ge 10$,
and (ii) $\chil^2(G)\le 5$ if $g\ge 14$,
and (iii) $\chil^2(G)=4$ if $g\ge 24$ (note that (iii) was also
proved in~\cite{BIN07}, mentioned above).
Their first result was strengthened by
Cranston and Kim~\cite{CranstonK08} and also Havet~\cite{Havet09}; both groups
showed that if $g\ge 9$, then $\chil^2(G)\le 6$.  
This was recently strengthened further by Kim and Luo~\cite{KL25}, who showed that 
$\chil^2(G)\le 6$ when $G$ has no cycles of lengths 4 through 8 (but 3-cycles are allowed).
The second result of~\cite{DvorakST08} was
strengthened by Havet, who showed that if $g\ge 13$, then $\chil^2(G)\le 5$.
Since always $\chi^2(G)\ge \Delta+1$, no restriction on girth implies
$\chi^2(G)\le 3$.  In particular, $\chi^2(K_{1,3})=4$.
For result (iii), Borodin and Ivanova slightly weakened the hypothesis for 
coloring (rather than list coloring) needed to guarantee $\chi^2(G)=4$.
In 2010 they showed~\cite{BorodinI11DAIO} that $g\ge 23$ suffices; 
in 2012 they strengthened~\cite{BorodinI12DMGT} this to $g\ge 22$. 
In 2021 La and Montassier~\cite{LM-girth21} extended this to $g\ge 21$.

In addition to the results mentioned above, various authors have conjectured that $\chi^2(G)\le 6$
if $G$ is planar and subcubic with some additional cycle length forbidden.  For easy reference, we gather
all of these statements into a single conjecture.

\begin{conjecture}[\cite{DvorakST08,HJT,FHS}]
	\label{subcubic-planar6-conj}
	If $G$ is planar and subcubic, then $\chi^2(G)\le 6$ when either (i) $G$ has no 
	3-cycle~\cite{DvorakST08}, (ii) $G$ can be drawn in the plane with no 5-face~\cite{HJT}, 
	or (iii) $G$ is bipartite~\cite{DvorakST08,FHS}.
\end{conjecture}
\noindent
Relatively little progress has been made on Concecture~\ref{subcubic-planar6-conj}.  However, \cite{FHS} proved 
that $\chi^2(G)\le 4$ if $G$ is planar and subcubic with each face having length a multiple of 4.  And for a 
slightly more general class of graphs they proved the bound $\chi^2(G)\le 6$.

To conclude this section,
we remark briefly about squares of cycles.
Note that $(C_5)^2=K_5$, so $\chi^2(C_5)=\chil^2(C_5)=\chiAT^2(C_5)=5$.
Now consider $C_k$ for some $k\ne 5$.  If $3|k$, then
$\chiAT^2(C_k)=3$; otherwise $\chiAT^2(C_k)=4$.  Proving the bound for coloring is
easy.  The bound $\chiAT^2(C_k)\le 4$ follows from Brooks'
Theorem~\cite{HladkyKS10}.  When $3\nmid k$, the independence number of $C_k^2$ is
less than $k/3$, so $\chi^2(C_k)=4$.  When $6|k$, Juvan, Mohar, and
\v{S}krekovski~\cite{JuvanMS98} showed that $\chil^2(C_k)=3$.  Using ideas
from~\cite[Lemma 16]{CranstonR13paint} this result can be strengthened to prove
that if $3|k$, then $\chiAT^2(C_k)=3$. 

\subsection{Graph Classes with Bounded Maximum Average Degree}
\label{sec:mad}

Nearly all results in Sections~\ref{sec:wegner} and
\ref{sec:planar-girth} are
proved using the discharging method.  (One notable exception is 
%the result of Havet, van den Heuvel, McDiarmid, and Reed~\cite{HavetHMR}, 
Theorem~\ref{HvdHMR-thm}, which proves
Wegner's conjecture asymptotically; it instead uses the probabilistic method,
which we discuss at the end of Section~\ref{erdos-nesetril-sec}.)
Discussing this technique in depth is outside the scope of this survey. 
However, for the interested reader, we recommend \emph{An Introduction to the
Discharging Method via Graph Coloring}~\cite{CranstonW-guide}.

Many discharging proofs for planar graphs, particularly sparse planar graphs,
use planarity in only a very weak sense: to bound the number of edges in any induced
subgraph.  This observation leads to the notion of \emph{maximum average
degree}, denoted \EmphE{$\mad(G)$}{-.5mm}, and defined as
$
{\mad(G)} := \max_{H\subseteq G, |V(H)|\ge 1}2|E(H)|/|V(H)|.
$
Forests are precisely the class of graphs $G$ with $\mad(G)<2$.  
Since every $n$-vertex planar graph has at most $3n-6$ edges, it
has average degree less than 6.  Since planar graphs are hereditary, 
every planar graph $G$ has $\mad(G)<6$.  More generally, we have the following
observation.

\begin{lemma}[Folklore]
If $G$ is a planar graph with girth at least $g$, then $\mad(G)<\frac{2g}{g-2}$.
\label{lem:folklore}
\end{lemma}
\begin{proof}
We sum the sizes of all $|F(G)|$
faces, which gives the inequality $2|E(G)|\ge g|F(G)|$.  Substituting this
inequality into Euler's formula, and solving for $\ad(G)$, the average degree of $G$,
yields $\ad(G)<\frac{2g}{g-2}$.  Each subgraph $H$ is also planar, with
girth at least $g$, so we get the desired result.
\end{proof}

Now $\mad(G)$ gives a natural way to strengthen and refine results about coloring
planar graphs, specifically with high girth.  One advantage of 
using $\mad$ is that we are no longer restricted to those values of $\mad$
that arise from Lemma~\ref{lem:folklore} when $g$ is an integer.

\captionsetup{width=.75\linewidth}
\renewcommand{\arraystretch}{1.4}
\begin{figure}[!ht]
\centering
\begin{tabular}{c|cccccccc}
$\Delta $ & 4 & 5 & $\Delta_{\epsilon}$ & 17 & $\Delta_{\epsilon,c}$ & 8 \\ \hline
$\mad(G)$ & $16/7$ & $2+\frac{4\Delta-8}{5\Delta+2}$ & $3-\epsilon$ & 3 &
$4-\frac{4}{c+1}-\epsilon$ & 4 \\
%\hline
$\chil^2$ & 
$\Delta+1$ & 
$\Delta+1$ & 
$\Delta+1$ & 
$\Delta+2$ & 
$\Delta+c$ & 
$3\Delta+1$
\end{tabular}
\caption{Triples $(D,k,C)$ such that every graph $G$ with 
$\Delta(G)\ge D$ and $\mad(G)<k$ 
satisfy $\chil^2(G)\le C$. References appear below (and near the end of the
section).}
%graph with maximum degree $\Delta=3$ and girth at least $g_k$.  (The entry $(3,-)$
%denotes that no such pair $(3,g_3)$ exists. References are in the preceeding few
%paragraphs.)}
\end{figure}

A natural first question, in view of Section~\ref{sec:planar-girth},
is whether
Conjectures~\ref{conj:wang-lih} and~\ref{conj:girth5} remain true when the
hypotheses are relaxed to $\mad(G)<\frac{2g}{g-2}$.  Dolama and
Sopena~\cite{DolamaS} proved that $\chi^2(G)=\DeltaG+1$ when $\DeltaG\ge 4$
and $\mad(G)<\frac{16}7$, which includes planar graphs with girth at least 16;
see also~\cite{CES}.  Cranston and
\v{S}krekovski~\cite{CranstonS} proved that $\chil^2(G)=\DeltaG+1$ whenever
$\DeltaG\ge 5$ and $\mad(G)<2+\frac{4\DeltaG-8}{5\DeltaG+2}$.
In particular, this includes all planar graphs with girth at least
$7+\frac{12}{\DeltaG-2}$.  Hence, it proves the analogue of
Conjecture~\ref{conj:wang-lih} for $g\ge 8$.  The case $g\ge 7$ was subsumed by
Bonamy, L\'{e}v\^{e}que, and Pinlou~\cite{BLP1} and strengthened further in~\cite{BLP2}.  
They proved that for each $\epsilon>0$ there exists $\Delta_\epsilon$ such that
if $G$ is a graph with $\mad(G)<3-\epsilon$ and $\DeltaG\ge \Delta_\epsilon$,
then $\chil^2(G)=\DeltaG+1$.  
% this next line seems to be wrong
%As $\epsilon\to 0$, the value of $\Delta_\epsilon$ asymptotically approaches $\frac{40}{\epsilon}$.
So Conjecture~\ref{conj:wang-lih}, with girth at least 7, holds in a much more
general context, requiring only $\mad<3-\epsilon$, for any $\epsilon>0$.
Note that this is sharp, since, as we saw in Proposition~\ref{prop:girth6}, 
there exist planar graphs $G$ with girth 6 (and hence $\mad<3$) such that
$\chi^2(G)=\DeltaG+2$.  One of the prettiest results in this area is the
following, due to Bonamy, L\'{e}v\^{e}que, and Pinlou~\cite{BLP3}.

\begin{theorem}[\cite{BLP3}]
\label{BLP3-thm}
If $G$ is a graph with $\mad(G)<3$ and $\DeltaG\ge 17$, then $\chil^2(G)\le
\DeltaG+2$.
\label{thm:BLP17}
\end{theorem}

This theorem generalizes that of Dvo\v{r}\'{a}k et~al.\ for girth 6, and also its
strengthenings by Borodin et~al., mentioned above.
At this point, it would be natural to guess that Theorem~\ref{thm:BLP17} extends
to the more general case $\mad(G)<\frac{10}3$ (the bound for planar graphs with
girth 5), at least for $\DeltaG$ sufficiently large.  But this generalization
is false.  Charpentier~\cite{charpentier, BLP2} constructed, for each positive
integer $C$, a family of graphs $G$ with unbounded maximum degree that each
has $\mad(G)<4-\frac{2}{C+1}$, and $\chi^2(G)=\DeltaG+C+1$.  Taking $C=2$
yields a class of counterexamples to our purported generalization. 

(Our proof uses the following easy observation. A fractional orientation of a
graph $G$ with maximum indegree $\alpha$, for any
$\alpha>0$, proves that $\mad(G)\le 2\alpha$.  To see this, note that the
maximum in the definition of $\mad(G)$ can be restricted to induced subgraphs
$H$.  For each such $H$, we have $|E(H)|= \sum_{v\in V(H)}d^+_H(v)\le
|V(H)|\alpha$.  Thus, $\mad(G)\le 2\alpha$.)

\begin{figure}[!t]
\centering
\begin{tikzpicture}[thick, scale=1.30]
\tikzstyle{uStyle}=[shape = circle, minimum size = 5.5pt, inner sep = 0pt,
outer sep = 0pt, draw, fill=white, semithick]
\tikzstyle{usStyle}=[shape = circle, minimum size = 4.5pt, inner sep = 0pt,
outer sep = 0pt, draw, fill=black, semithick]
\tikzset{every node/.style=uStyle}

\clip (-2,7.5) rectangle + (10.2,-7.1);

\foreach \name/\x/\y in {v1/0/5, v2/0/4, vD/0/1, w1/6/5, w2/6/4, wC/6/1, x/-1/3, y/7/3}
\draw (\x,\y) node (\name) {};
\draw (v1) -- (w1) node[pos = .5, usStyle] {} 
-- (v2) node[pos = .35, usStyle] {}
-- (w2) node[pos = .50, usStyle] {} 
-- (v1) node[pos = .65, usStyle] {} (vD)
-- (wC) node[pos = .50, usStyle] {} (v1)
-- (wC) node[pos = .35, usStyle] {}
-- (v2) node[pos = .65, usStyle] {} (w1) 
-- (vD) node[pos = .35, usStyle] {}
-- (w2) node[pos = .65, usStyle] {}
(vD) -- (x) -- (v2) (x) -- (v1) 
(wC) -- (y) -- (w2) (y) -- (w1) 
;
\draw[dotted, thick] (v2) ++ (0,-.5cm) -- (0,1.5cm);
\draw[dotted, thick] (w2) ++ (0,-.5cm) -- (6,1.5cm);
\draw (x) edge[bend left, out = 90, in = 90, looseness=1.75] (y);

\foreach \lab\where in {v_1/v1, v_2/v2, w_1/w1, w_2/w2}
\draw (\where) ++ (0,3mm) node[draw=none, fill=none] {\footnotesize{$\lab$}};
\foreach \lab\where in {v_C/vD, w_D/wC, x~/x, ~y/y}
\draw (\where) ++ (0,-3mm) node[draw=none, fill=none] {\footnotesize{$\lab$}};
\end{tikzpicture}
\captionsetup{width=.65\linewidth}
\caption{A graph $G_{C,D}$ with $\Delta(G_{C,D})=D+1$ and $\chi^2(G_{C,D})\ge
\Delta(G_{C,D})+C+1$ and $\mad(G_{C,D})<4-\frac{2}{C+1}$, as in
Theorem~\ref{charpentier-thm}.\label{charpentier-fig}}
\end{figure}

\begin{theorem}[\cite{BLP2}]
For all positive integers $C$ and $D$ with $C\le D$, there exists a graph
$G_{C,D}$ with
$\Delta(G_{C,D})=D+1$ and $\chi^2(G_{C,D})\ge \Delta(G_{C,D})+C+1$ and
$\mad(G_{C,D})<4-\frac{2}{C+1}$.
\label{charpentier-thm}
\end{theorem}
\begin{proof}
To form $G_{C,D}$, begin with a complete bipartite graph $K_{C,D}$ with parts
$\{v_1,\ldots,v_C\}$ and $\{w_1,\ldots,w_D\}$; see
Figure~\ref{charpentier-fig}.  Add new vertices $x$ and
$y$, where $N(x):=\{v_1,\ldots,v_C,y\}$ and $N(y):=\{w_1,\ldots,w_D,x\}$.
Finally, subdivide each edge $v_iw_j$ with a new 2-vertex $z_{ij}$.  It is easy
to check that in $G^2_{C,D}$ all of vertices $v_1,\ldots,v_C,w_1,\ldots,w_D,x,y$
form a clique, so $\chi^2(G_{C,D})\ge C+D+2$.  Further, $\Delta(G_{C,D})=D+1$,
so $\chi^2(G_{C,D})\ge \Delta(G_{C,D})+C+1$.

Now we must verify that $\mad(G_{C,D})<4-\frac2{C+1}$.  We fractionally orient
$E(G_{C,D})$ as follows.  
Orient each edge $v_iz_{ij}$ towards $z_{ij}$.  
Orient each $z_{ij}w_j$ with the fraction $\frac1{C+1}$ towards $w_j$ and the
rest towards $z_{ij}$.  
Orient $v_ix$ towards $v_i$ and orient $w_jy$ towards $w_j$.
Orient $xy$ arbitrarily.  This orientation has maximum indegree $2-\frac1{C+1}$,
which proves that $\mad(G_{C,D})\le 2(2-\frac1{C+1})$.  Showing the inequality
is strict requires a few more details, which we omit.
\end{proof}

Bonamy, L\'{e}v\^{e}que, and Pinlou~\cite{BLP2} suggested that Charpentier's construction is sharp.
They asked whether for each $\epsilon>0$ there exists a constant $\Delta_\epsilon$ such
that if $\mad(G)<4-\frac2{c+1}-\epsilon$ and $\DeltaG>\Delta_\epsilon$, then
$\chil^2(G)\le\DeltaG+c$.  (Their result above for
$\mad<3-\epsilon$ proves this for $c=1$.)  In this direction, they
showed~\cite{BLP1} that if $\mad(G)\le 4 - \frac{40}{c+16}$, then $\chil^2(G)\le
\DeltaG+c$.  For $c\ge 3$, Yancey~\cite{Yancey} strengthened this bound.  He
showed that, for each $c\ge 3$ and $\epsilon > 0$, if $\DeltaG$ is
sufficiently large (in terms of $c$ and $\epsilon$) and
$\mad(G)<4-\frac{4}{c+1}-\epsilon$, then $\chil^2(G)\le
\DeltaG+c$.

It is worth noting, in the question above of Bonamy et al., that it is
essential to bound $\mad(G)$ strictly away from 4. 
Charpentier~\cite{charpentier} also posed the following conjecture: There exists a constant
$D$ such that if $G$ has $\Delta\ge D$ and $\mad(G)<4$, then $\chi^2(G)\le
2\Delta$.  This was first disproved by Kim and Park~\cite{KP_charpentier-CE}.
Using ideas similar to those in the proof of
Theorem~\ref{charpentier-thm}, they constructed graphs $G_{\Delta}$ with
$\mad(G_{\Delta})<4$, and maximum degree $\Delta$, but $\chi^2(G_{\Delta})=2\Delta+2$.
The current best known construction, which is due to Hocquard, Kim, and 
Pierron~\cite{HKP} and presented below, has $\chi^2(G_{\Delta})=\frac52\Delta$.

\begin{figure}[!ht]
\centering
\begin{tikzpicture}[scale=1]
\tikzstyle{uStyle}=[shape = circle, minimum size = 5.5pt, inner sep = 0pt,
outer sep = 0pt, draw, fill=white, semithick]
\tikzstyle{lStyle}=[shape = circle, minimum size = 4.5pt, inner sep = 0pt,
outer sep = 0pt, draw, fill=none, draw=none]
\tikzstyle{usStyle}=[shape = circle, minimum size = 3.5pt, inner sep = 0pt,
outer sep = 0pt, draw, fill=black, semithick]
\tikzstyle{usGStyle}=[shape = circle, minimum size = 3.5pt, inner sep = 0pt,
outer sep = 0pt, draw, fill=gray, semithick]
\tikzset{every node/.style=uStyle}
\def\rad{4.5cm}

\foreach \i in {1,2,3,4,5}
\draw (\i*72+18:\rad) node[] (x\i) {};
\draw (0,0) node[draw=none] (origin) {};

% ``outside edges''
\foreach \i/\j in {1/2, 2/3, 3/4, 4/5, 5/1}
{
\draw (x\i) -- (barycentric cs:x\i=5,x\j=5,origin=-.5) node[usStyle] {} -- (x\j);
\draw (x\i) -- (barycentric cs:x\i=5,x\j=5,origin=0) node[usStyle] (y\i) {} -- (x\j);
\draw (x\i) -- (barycentric cs:x\i=5,x\j=5,origin=.55) node[usStyle] {} -- (x\j);
}

% ``diagonals''
\foreach \i/\j in {1/3, 2/4, 3/5, 4/1, 5/2}
{
\draw (x\i) -- (barycentric cs:x\i=3,x\j=3,origin=-.625) node[usStyle] {} -- (x\j);
\draw (x\i) -- (barycentric cs:x\i=3,x\j=3,origin=0) node[usStyle] (z\i) {} -- (x\j);
\draw (x\i) -- (barycentric cs:x\i=3,x\j=3,origin=.785) node[usStyle] {} -- (x\j);
}

\foreach \i/\j in {1/3, 2/4, 3/5, 4/1, 5/2}
{
\draw (y\j) edge[bend left=50] (y\i);
\draw (barycentric cs:y\i=3,y\j=3,origin=-3.51) node[usGStyle] {};
\draw (y\i) edge[bend left=45] (z\j);
}

%%% Shift these vertices along these edges!!
\foreach \i/\j/\k in {1/3/5, 2/4/1, 3/5/2, 4/1/3, 5/2/4}
\draw (barycentric cs:y\i=3,z\j=3.35,y\k=-1.07,origin=2.3) node[usGStyle] {};

\foreach \i/\j in {1/2, 2/3, 3/4, 4/5, 5/1}
{
	\draw (z\i) edge[bend right=10] (z\j);
	\draw (barycentric cs:z\i=3,z\j=3,origin=-.5) node[usGStyle] {};
}

\end{tikzpicture}
\captionsetup{width=.485\textwidth}
\caption{The graph $G_D$ in Theorem~\ref{HKP-thm}.  In $G_D^2$, the black vertices form a clique of
order $5D$.\label{HKP-thm-fig}}
\end{figure}

\begin{theorem}[\cite{HKP}]
For each positive even integer $D$, there exists a 2-degenerate graph $G_D$,
with $\mad(G_D) < 4$, maximum degree $D$, and $\chi^2(G_D)\ge \frac{5D}2$.
(See Figure~\ref{HKP-thm-fig}.)
\label{HKP-thm}
\end{theorem}
\begin{proof}
Begin with a copy of $K_5$ and replace each edge $vw$ with a copy of
$K_{2,D/2}$, identifying the vertices in the small part with $v$ and $w$.
Call these $10(D/2)=5D$ added vertices \emph{new vertices}.
Now for each pair of new vertices, $x,y$, that is at distance 3 or more in the
resulting graph, add an additional vertex $z_{xy}$ adjacent to only $x$ and $y$.
	(To keep the figure looking clean, we have only drawn these gray vertices $z_{xy}$ 
	when vertices $x$ and $y$ are both in the middle of their copy of $K_{2, D/2}$. But
	the construction includes a gray vertex for \emph{every} such pair $x,y$.)
Call the resulting graph $G_D$.  Note that in $G_D^2$ the $5D$ new vertices
form a clique.  It is easy to check that $G_D$ is 2-degenerate (so has $\mad <
4$) and also that $\Delta(G_D)=4(D/2)=2D$.  Thus, $\chi(G_D^2)\ge \omega(G_D^2) \ge
5D = \frac52\Delta(G_D)$, as desired.
\end{proof}

To complement Theorem~\ref{HKP-thm}, the same authors~\cite{HKP} show that if
$\Delta\ge 8$ and $\mad(G)<4$, then $\chil^2(G)\le 3\Delta+1$.
They also posed the following two questions.

\begin{question}[\cite{HKP}]
\label{HKP-question1}
Is there an integer $D$ such that every graph $G$ with $\Delta(G)\ge D$ 
that is 2-degenerate satisfies $\chi^2(G)\le \frac52\Delta(G)$?
\end{question}

\begin{question}[\cite{HKP}]
\label{HKP-question2}
Is there an integer $D$ such that every graph $G$ with $\Delta(G)\ge D$ and
$\mad(G)<4$ satisfies $\chi^2(G)\le \frac52\Delta(G)$?
\end{question}

If we aim to answer Question~\ref{HKP-question1} or~\ref{HKP-question2} negatively,
then it is natural to try to construct some infinite family of graphs $H_D$ such that 
$\Delta(H_D)=D$ and $\omega(H^2_D)>\frac52$.  But Cranston and Yu~\cite{CY} showed that
the construction of Theorem~\ref{HKP-thm} is best possible, up to an additive constant.
They proved that if $G$ is $2$-degenerate, then $\omega(G^2)\le \frac52\Delta(G)+72$.
And if $\mad(G)<4$, then $\omega(G)\le \frac52\Delta(G)+532$.  The first of these results
was later sharpened by Kim and Lian~\cite{KL}.  They showed that there exists $D$ such 
that if $G$ is $2$-degenerate with $\Delta(G)\ge D$, then $\omega(G^2)\le \frac52\Delta(G)$.
As a weaker version of Question~\ref{HKP-question2}, we explicitly conjecture the following.

\begin{conjecture}
\label{question3}
There exists an integer $D$ such that every graph $G$ with $\Delta(G)\ge D$ and
$\mad(G)<4$ satisfies $\omega^2(G)\le \frac52\Delta(G)$.
\end{conjecture}

\section{Strong Edge Coloring}
\subsection{The
\texorpdfstring{Erd\H{o}s}{Erdos}--\texorpdfstring{Ne\v{s}et\v{r}il}{Nesetril} Conjecture}
\label{erdos-nesetril-sec}

A \Emph{strong edge-coloring} of a graph $G$ colors each edge of $G$ such that edges
get distinct colors if either (i) they share a common endpoint or (ii) they each
share an endpoint with the same edge.  Equivalently, it is a proper coloring of
the square of the line graph of $G$.  The \emph{strong edge-chromatic number} of
$G$, denoted \Emph{$\chi^s(G)$}, is the smallest number of colors that allow a strong
edge-coloring.  At a seminar in Prague in 1985, \erdos\ and \nesetril\
posed~the~following (see~\cite{FaudreeGST89}).

\begin{conjecture}[\cite{FaudreeGST89}] %[Erdos--Nesetril]
\label{erdos-nesetril-conj}
If $G$ is a graph with maximum degree $\Delta$, then
$$
\chi^s(G)\le
\left\{
\begin{array}{ll}
\frac54\Delta^2 & \mbox{if $\Delta$ is even,}\\
\frac{5\Delta^2-2\Delta+1}4 & \mbox{if $\Delta$ is odd.}\\
\end{array}
\right.
$$
\end{conjecture}

The following example shows the conjecture is best possible.
When $\Delta$ is even, let $V_i$ be an independent set of size $\frac{\Delta}2$, for
each $i\in\{1,2,3,4,5\}$.  Let $V(G)=\cup_{i=1}^5V_i$ and for each $v\in V_i$ and $w\in
V_j$, add edge $vw$ to $E(G)$ if and only if $(i-j)\equiv 1\pmod 5$.  The
resulting graph $G$ is a ``5-cycle of independent sets''; since each pair of the
$\frac54\Delta^2$ edges lies together on a 4-cycle or 5-cycle, all edges must
get distinct colors.  When $\Delta$ is odd, the construction is similar.  Now
let $V_1$ and $V_2$ have size $\frac{\Delta+1}2$ and each remaining $V_i$ have
size $\frac{\Delta-1}2$.  In this case the graph has $\frac{5\Delta^2-2\Delta+1}4$
edges, and again all edges must get distinct colors.  
Chung, Gy{\'a}rf{\'a}s,
Tuza, and Trotter~\cite{ChungGTT90} proved that this construction is the unique
worst case among graphs where all edges must get distinct colors.  %More precisely, 
Specifically, they showed that the maximum number of edges in a $2K_2$-free graph
is exactly given by the upper bound in the \erdos--\nesetril\ Conjecture;
further, the extremal graphs are unique.

When we cannot prove a conjectured upper bound on the chromatic number, 
it is natural to seek to prove the corresponding bound on its clique 
number.  (We did this for squares of planar graphs near the start of 
Section~\ref{sec:wegner}, and for squares of graphs $G$
with $\mad(G)<4$ in Conjecture~\ref{question3} and the paragraph preceding it.)
We write $\omega^s(G)$ for the clique number of the square of the line graph of $G$.
Conjecture~\ref{erdos-nesetril-conj} would imply that if $G$ has maximum degree $\Delta$,
then $\omega^s(G)\le \frac54\Delta^2$.
In this direction, \'{S}leszy\'{n}ska-Nowak~\cite{SN} proved that $\omega^s(G)\le \frac32\Delta^2$,
and Faron and Postle~\cite{FP} improved this to $\omega^s(G)\le \frac43\Delta^2$.
To further improve the bound on $\omega^s(G)$, Cames van Batenburg, Kang, and Pirot~\cite{CvBKP}
considered the problem when various subgraphs are forbidden.  They proved that 
(i) $\omega^s(G)\le \frac54\Delta^2$ if $G$ is triangle-free, (ii) $\omega^s(G)\le \Delta^2$ if 
$G$ is $C_5$-free, and (iii) $\omega^s(G)\le \Delta^2$ if $G$ is $C_{2k+1}$-free, with $k\ge 3$
and $\Delta\ge 3k^2+10k$. (By \emph{$H$-free}, here we forbid all copies of subgraph $H$, not necessarily 
induced.)  When $H$ is an even cycle, they posed the following.

\begin{conjecture}[\cite{CvBKP}]
	\label{CvBKP-conj}
	If $G$ is $C_{2k}$-free with $\Delta\ge \max\{4,2k-2\}$, then $\omega^s(G)\le (2k-1)(\Delta-k+1)$.  
\end{conjecture}
If true, Conjecture~\ref{CvBKP-conj} is sharp: start with $K_{2k-1}$ and add $\Delta-2k+2$ pendant 
edges at each clique vertex.  The authors of~\cite{CvBKP} proved partial results in support of the 
conjecture, including proving it when $k=2$.
Cho, Choi, Kim, and Park~\cite{CCKP2} proved this conjecture up to an additive constant $O(k^2)$.
In the same paper, the authors also proved that if $G$ is $C_{2k}$-free and bipartite, 
then $\omega^s(G)\le k(\Delta-1)+1$.  This had been previously conjectured in~\cite{CvBKP},
and is sharp as follows: start with a copy of $K_{\Delta,k-1}$ and at some vertex in the big part
add $\Delta-k+1$ pendant edges.  

\bigskip

The \erdos--\nesetril\ Conjecture is easy for $\Delta\le2$.  For $\Delta=3$,
it was proved by
Andersen~\cite{Andersen92} and also by Hor\'{a}k, He, and
Trotter~\cite{HorakHT93}.  For $\Delta=4$, Hor\'{a}k~\cite{Horak90} showed that
$\chi^s(G)\le 23$, Cranston~\cite{Cranston06} strengthened this to
$\chi^s(G)\le 22$, and Huang, Santana, and Yu~\cite{HSY} strengthened it
further to $\chi^s(G)\le 21$.
The proofs of \cite{Andersen92}, \cite{Cranston06}, and~\cite{HSY} all follow the same
approach.  For some specified vertex $v$ in $G$, we color the edges greedily in
order of non-increasing distance from $v$.  Thus, at the time each edge is
colored, it has at most $2\Delta^2-3\Delta$ restrictions on its color, so it
uses color at most $2\Delta^2-3\Delta+1$.  This is true for all edges except
those incident to $v$.  So the hard work consists of showing that we can choose
$v$ such that we can finish the coloring.  Substituting $\Delta=3$ gives the
desired bound of 10.  Substituting $\Delta=4$ gives 21. 
To finish the coloring easily,~\cite{Cranston06} uses a $22^{\textrm{nd}}$ color, only near $v$.
To save this extra color,~\cite{HSY} considers a wide range of options that
work for $v$, and ultimately shows that every graph contains at least one of them.

The class $\G_7$ is all graphs $G$ with $\max_{vw\in E(G)}d(v)+d(w)\le 7$;
this generalizes subcubic graphs.  Intuitively, $\G_7$ lies between the classes of subcubic and
subquartic graphs, at least in terms of colors needed for a strong edge-coloring.  Chen, Huang,
Yu, and Zhou~\cite{CHYZ} conjectured $\chi^s(G)\le 13$ for each $G\in \G_7$ (and proved
$\chi^s(G)\le 15$).  This conjecture is best possible due to a blow-up of the 5-cycle, similar
to that described above.  The best progress is due to Nelsen and Yu~\cite{NY}, who proved the 
conjecture for all planar graphs.

For every bipartite graph $G$, Faudree, Gy{\'a}rf{\'a}s, Schelp, and
Tuza~\cite{FaudreeGST89} conjectured that $\chi^s(G)\le \DeltaG^2$.
They proved the weaker statement that if $G$ is such that in every strong
edge-coloring each edge must get a distinct color (and $G$ is bipartite), 
then $G$ has at most $\DeltaG^2$ edges.  
Brualdi and Quinn Massey~\cite{BQm} 
%strengthened the above conjecture, as follows. 
conjectured the following.

\begin{conjecture}[\cite{BQm}]
\label{BQm-conj}
If $G$ is bipartite with parts
$A$ and $B$, then $\chi^s(G)\le \Delta_A\Delta_B$, where $\Delta_A$ and
$\Delta_B$ are the maximum degrees, respectively, of vertices in $A$ and $B$.  
\end{conjecture}

This conjecture is sharp, as witnessed by the complete bipartite graph
$K_{\Delta_A,\Delta_B}$.  As a first step, Brualdi and Quinn Massey proved the
case of the conjecture when each cycle in $G$ has length divisible by 4.  
The case of Conjecture~\ref{BQm-conj} when $\Delta_A=1$ is trivial, since the
graph is a disjoint union of stars.  The case when $\Delta_A=2$ was proved by
Nakprasit~\cite{nakprasit08}, and the case when $\Delta_A=3$ by Huang, Yu,
and Zhou~\cite{HYZ}.

Mahdian~\cite{Mahdian}
used the R\"{o}dl nibble (more details are given at the end of this section) to 
show that if $G$ is $C_4$-free, then $\chi^s(G)\le (2+o(1))\frac{\DeltaG^2}{\ln \DeltaG}$.
More generally, his proof works for graphs that are $K_{2,t}$-free for any fixed $t$.
(This result was generalized by Vu~\cite{Vu} to a similar upper
bound, with a worse multiplicative constant, when any fixed bipartite graph 
is forbidden as a subgraph.) Mahdian also conjectured that the same bound holds
for $K_{t,t}$-free graphs, for any fixed $t$.  This conjecture was proved by 
Bi, Bradshaw, Dhawan, and Xu~\cite{BBDX}, using a variation of Mahdian's argument.
In fact, with more work they saved a factor of 2: Every $K_{t,t}$-free graph satisfies
$\chi^s(G)\le (1+o(1))\frac{\DeltaG^2}{\ln \DeltaG}$.

Mahdian also strengthened the conjecture of
Faudree et al.~in a different direction.  

\begin{conjecture}[\cite{Mahdian}]
\label{mahdian-conj}
If $G$ is a graph with no 5-cycle, then $\chi^s(G)\le \DeltaG^2$.
\end{conjecture}
\noindent

Greedy coloring (in any order) shows that $\chi^s(G)\le 2\Delta^2-2\Delta+1$.
\erdos\ and \nesetril\ specifically asked for $c>0$ such that
$\chi^s(G)\le(2-c)\Delta^2$ for all graphs $G$.  Molloy and
Reed~\cite{MolloyR-strong} provided such a $c$ by using the probabilistic
method; see also~\cite[Chapter~10]{MolloyR-GCPM}.

\begin{theorem}[\cite{MolloyR-strong}]
There exists $\Delta_0$ such that when $\Delta\ge \Delta_0$
we have $\chi^s(G)\le 1.998\Delta^2$.  
\label{thm:MR-strong}
\end{theorem}
\vspace{-.3in}

\begin{corollary}[\cite{MolloyR-strong}]
There exists $c>0$ such that $\chi^s(G)\le (2-c)\DeltaG^2$ for all $G$.
\end{corollary}
\begin{proof}
Let $c:=\min\{.002,\frac1{\Delta_0}\}$.  
Given a graph $G$, if
$\Delta\ge {\Delta_0}$, then $\chi^s(G)\le1.998\Delta^2\le(2-c)\Delta^2$.  Otherwise,
assume $\Delta<\Delta_0$.  Applying Brooks' Theorem to the square of the line
graph shows that $\chi^s(G)\le
2\Delta^2-2\Delta=(2-\frac2{\Delta})\Delta^2<(2-\frac1{\Delta_0})\Delta^2\le (2-c)\Delta^2$.
\end{proof}

The proof of Theorem~\ref{thm:MR-strong} is quite nice, so we outline it below.
We focus on the fact that there exist $\Delta_0$ and $c>0$ such that $\chi^s(G)\le
(2-c)\Delta^2$ for all $G$ with $\Delta\ge \Delta_0$, but we don't show that 
$c=.002$ suffices.

\begin{proof}[Proof sketch of Theorem~\ref{thm:MR-strong}]
For simplicity, assume that $G$ is regular; if not, then we embed it as
a subgraph in a regular graph with the same maximum degree $\Delta$.
Let $H$ be the square of the line graph of $G$. We will 
color $H$ using at most $(2-c)\Delta(G)^2$ colors, for some $c>0$.
Color each vertex of $H$ randomly (and uniformly) from the set
$\{1,\ldots,\Delta(G)^2\}$ and whenever two adjacent vertices get the
same color, uncolor them both.  

We can check that, for each vertex $v$ in $H$, the
expected number of colors that are repeated in the neighborhood of $v$ is at
least $c'\Delta(G)^2$, for some $c'>0$ when $\Delta(G)$ is sufficiently large.
(This uses that $H$ is the square of a line graph, so
$H[N_H(v)]$ has far fewer than $\binom{2\Delta(G)^2}{2}$ edges.)  By
Talagrand's inequality~(see~\cite[Chapter~10]{MolloyR-GCPM}) we can show that,
for each vertex $v\in V(H)$, the
number of colors repeated in its neighborhood is close to the expected value
with high probability.  By the {\lovasz} Local Lemma, there exists some random
coloring such that the number of colors repeated in \emph{every} neighborhood is close
to its expected value.  Given such a partial coloring, we can complete the
coloring greedily.  This proves the desired result.
\end{proof}

Recall that each strong edge-coloring of a graph $J$ corresponds to a coloring of the 
square of the line graph of $J$.  And notice that if $G$ is the line graph of a simple
graph $J$ (with $\Delta(J)\ge 3$), then $\omega(G) = \Delta(J)$.  This prompted de Joannis de Verclos,
Kang, and Pastor~\cite{dJdVKP} to posit the following generalization of Conjecture~\ref{erdos-nesetril-conj}.

\begin{conjecture}[\cite{dJdVKP}]
\label{claw-free-conj}
If $G$ is a claw-free graph, then
$$
\chi(G^2)\le
\left\{
\begin{array}{ll}
\frac54\omega(G)^2 & \mbox{if $\omega(G)$ is even,}\\
\frac{5\omega(G)^2-2\omega(G)+1}4 & \mbox{if $\omega(G)$ is odd.}\\
\end{array}
\right.
$$
\end{conjecture}

As evidence in support of Conjecture~\ref{claw-free-conj}, the authors of~\cite{dJdVKP} proved that 
there exists $c>0$ such that $\chi(G^2)\le (2-c)\omega(G)^2$.  (A graph $G$ is \emph{quasiline} if 
for each $v\in V(G)$ the set $N(v)$ can be covered by $2$ cliques.)
Their proof breaks into 3 stages: 
(1) showing that the result holds for all claw-free graphs if it holds for all quasiline graphs,
(2) showing that it holds for all quasiline graphs if it holds for all line graphs of multigraphs, and
(3) proving it for all line graphs of multigraphs.  The reductions in the first 2 stages both also work
for significantly larger values of $c$.  But the third phase is the current bottleneck.  Its proof is 
similar to the proof sketch of Theorem~\ref{thm:MR-strong}; but now we consider line graphs of 
\emph{multigraphs}, whereas Theorem~\ref{thm:MR-strong} considers only line graphs of simple graphs.

As more evidence for Conjecture~\ref{claw-free-conj}, Cames van Batenburg and Kang~\cite{CvBK} proved
the conjecture in the case that $\omega(G)=3$.  And in the case that $\omega(G)=4$, they proved it 
up to an additive error term of $+1$.  They also showed that Conjecture~\ref{claw-free-conj}
holds when $\omega(G)=4$ if and only if Conjecture~\ref{erdos-nesetril-conj} holds when $\Delta(G)=4$.  
(That is, when $\omega(G)=4$ the most difficult case of Conjecture~\ref{claw-free-conj} is when $G$ is a 
line graph.)
\bigskip

The proof of Theorem~\ref{thm:MR-strong} breaks the problem of bounding $\chi^s$
into two subproblems: (i) proving an upper bound on the density of the subgraph
induced by the neighborhood of any vertex in $H$ and (ii) using this density upper bound
to prove an upper bound on $\chi(H)$, since $\chi(H)=\chi^s(G)$.  Recently, a
string of papers has followed this approach, each bounding $\chi^s(G)$ when
$\DeltaG$ is sufficiently large.  First, Bruhn and Joos~\cite{BruhnJ} showed that 
$\chi^s(G)\le 1.93\DeltaG^2$.  
Bonamy, Perrett, and Postle~\cite{BPP} strengthened this to $\chi^s(G)\le 1.835\DeltaG^2$.
\label{1.772-anchor} Most recently Hurley, de Joannis de Verclos, and
Kang~\cite{HdjdVK} proved that $\chi^s(G)\le 1.772\DeltaG^2$.  
Davies, Kang, Pirot, and Sereni~\cite{DKPS} also developed a general framework for
deriving coloring bounds from bounds on the sparseness of neighborhoods.
Their result has several applications, including Johansson's famous bound 
on the chromatic number of triangle-free graphs, but it does not seem to apply 
for squares of line graphs.  

Recall that Molloy and Reed, in the proof of Theorem~\ref{thm:MR-strong},
constructed their partial coloring of $H$ by coloring each vertex uniformly
at random from $\{1,\ldots,\Delta(G)^2\}$ and then uncoloring every vertex that
got the same color as some neighbor.
This approach is called the Na\"{i}ve Coloring Procedure.  It was introduced by
Kahn~\cite{Kahn} in his proof that the list-coloring conjecture is true
asymptotically, and Molloy and Reed~\cite{MolloyR-GCPM} presented many other examples
of this technique.  The improvements on Theorem~\ref{thm:MR-strong} given in
\cite{BruhnJ}, \cite{BPP}, and~\cite{HdjdVK} all rely crucially on this method.
In~\cite{BruhnJ}, rather than uncoloring every vertex that got the same color as
some neighbor, Bruhn and Joos flipped a coin for each such conflict, and only
uncolored the vertices that lost at least one coin flip.  They also proved
stronger upper bounds on the density of each subgraph $G^2[N(v)]$ and proved
stronger bounds on the number of colors repeated in each neighborhood using 
Inclusion-Exclusion. % principle.  

In~\cite{BPP}
and~\cite{HdjdVK}, the authors iterated the Na\"{i}ve Coloring Procedure,  
%That is, they 
repeatedly using it to color a small fraction of the uncolored
vertices.  At each iteration, the number of colors available for each uncolored
vertex $v$ decreased; but the number of uncolored neighbors of each such $v$
decreased faster.  (For each $v$, this property holds with high
probability, so the %{\lovasz} 
Local Lemma guarantees that for some partial coloring
this property holds for all vertices, simultaneously.) Finally, each vertex
remaining uncolored had fewer uncolored neighbors than available colors, so the
partial coloring could be finished greedily.  

This iterative approach is often called the semi-random method or the R\"{o}dl Nibble
(each iteration is a nibble).  The R\"{o}dl Nibble has been applied successfully
to a wide range of problems in combinatorics.  Presenting further details is
beyond the scope of this survey.  However, we direct the interested reader to a
recent survey~\cite{nibble-survey} on this topic, as well as to~\cite[Section
V]{MolloyR-GCPM}.  Before leaving this topic, we note that variations of this
Na\"{i}ve Coloring Procedure also play central roles in the proofs of
Theorems~\ref{HvdHMR-thm} and~\ref{thm:MolloyR-total}.

\subsection{Subcubic Graphs}
\label{strong-subcubic-sec}
Within the subject of strong edge-coloring,
another line of research has focused on the case $\Delta=3$.  Since
$\chi^s(G)\le 10$ for all such graphs, we seek conditions to imply
$\chi^s(G)\le 9$ (resp. 8, 7, 6, and 5).   
Faudree, Schelp, Gy\'{a}rf\'{a}s, and Tuza~\cite{FaudreeSGT90} posed the
following nice sequence of conjectures for all subcubic graphs $G$.  For easy
reference, we present it as a single conjecture with six parts; 
the first 4 of these have been confirmed, the 6th was recently disproved, and the 5th
is still open.

\begin{conjecture}[\cite{FaudreeSGT90}]
\label{FSGT-conj}
\label{6part-conj}
If $G$ is a graph with $\Delta=3$, then the following six bounds hold.
\begin{itemize}
\setlength\itemsep{.05em}
\item[(a)] $\chi^s(G)\le 10$.  (Andersen~\cite{Andersen92} and Horak,
He, and Trotter~\cite{HorakHT93}.)
\item[(b)] $\chi^s(G)\le 9$ if $G$ is bipartite.  (% Confirmed by
Steger and Yu~\cite{StegerY93}.)
\item[(c)] $\chi^s(G)\le 9$ if $G$ is planar. (%Confirmed by 
Kostochka, Li, Ruksasakchai, Santana, Wang, and Yu~\cite{KostochkaLRSWY}.)
\item[(d)] $\chi^s(G)\le 6$ if $G$ is bipartite and no 3-vertices are adjacent.
(%Confirmed by 
Wu and Lin~\cite{WuL08}.)
\item[(e)] $\chi^s(G)\le 7$ if $G$ is bipartite with no 4-cycle.
\item[(f)] $\chi^s(G)\le 5$ if $G$ is bipartite with girth sufficiently large.
(Disproved by Lu\v{z}ar, Ma\v{c}ajov\'{a}, \v{S}koviera, and
Sot\'{a}k~\cite{LMSS}; also by Cranston~\cite{cranston-FSGT-ce}.)
\end{itemize}
\end{conjecture}

\begin{figure}[!h]
\centering
\begin{tikzpicture}[thick, scale=.8]
\tikzstyle{uStyle}=[shape = circle, minimum size = 5.5pt, inner sep = 0pt,
outer sep = 0pt, draw, fill=white, semithick]
\tikzset{every node/.style=uStyle}

\begin{scope}[yscale=.8]
\foreach \x/\y/\lab in {0/0/A, 0/1/B, 1/0/C, 1/1/D, 2/1.8/E, .5/2.6/F, -1/1.8/G}
\draw (\x,\y) node (\lab) {};

\draw (C) -- (A) -- (D) -- (B) -- (C) -- (E) -- (F) -- (G) -- (A) (G) -- (B) (D) -- (E);
\end{scope}
% (a)

\begin{scope}[xshift = 2.25in, yshift = -.4in, yscale=1.05, xscale=1.1, rotate=90]
\foreach \top in {1,2,3}
\draw (\top,2) node (t\top) {};
\foreach \bottom in {1,2,3}
\draw (\bottom,1) node (b\bottom) {};

\foreach \top in {1,2,3}
{
\foreach \bottom in {1,2,3}
  \draw (t\top) -- (b\bottom);
}
\end{scope}
% (b)

\begin{scope}[xshift=2.5in, yshift=1, yscale=.705, xscale=1.1]
\foreach \x/\y/\lab in {0/0/A, 0/3/B, 1/0/C, 1/3/D, .5/1/E, .5/2/F}
\draw (\x,\y) node (\lab) {};

\draw (A) -- (B) -- (D) -- (C) -- (A) -- (E) -- (F) -- (B) (F) -- (D) (E) -- (C);

\end{scope}
% (c)

\begin{scope}[xshift=3.90in, yshift=1, yscale=.705, xscale=0.8]
\foreach \x/\y/\lab in {0/0/A, -1/1.5/B, 0/1.5/C, 1/1.5/D, 0/3/E}
\draw (\x,\y) node (\lab) {};

\draw (A) -- (B) -- (E) -- (C) -- (A) -- (D) -- (E);
\end{scope}
% (d)

\begin{scope}[xshift=5.20in, yshift=.425in, scale=1.1]
\tikzstyle{blavert}=[circle, draw, fill=black, inner sep=0pt, minimum width=6pt,semithick]
\tikzstyle{redvert}=[circle, draw, fill=white, inner sep=0pt, minimum width=6pt,semithick]
\draw \foreach \i in {2, 4, ..., 14}
{
(\i*360/14:1) node[blavert]{} -- (\i*360/14+360/14:1)
(\i*360/14+360/14:1) node[redvert]{} -- (\i*360/14+720/14:1)
(\i*360/14:1)  -- (\i*360/14+5*360/14:1)
};
\end{scope}
% (e)

\begin{scope}[xshift=6.25in, yshift=1, yscale=.705, xscale=1.1]
\foreach \x/\y/\lab in {0/0/A, 0/3/B, 1/0/C, 1/3/D, .5/1/E, .5/2/F}
\draw (\x,\y) node (\lab) {};

\draw (A) -- (E) -- (F) -- (B) (F) -- (D) (E) -- (C);

\end{scope}
% (f)

\foreach \x/\lab in {0.5/a, 4.05/b, 6.92/c, 9.90/d, 13.2/e, 16.4/f}
\draw (\x,-.8) node[draw=none] (\lab) {\footnotesize{(\lab)}};

\end{tikzpicture}
\captionsetup{width=.9\linewidth}
\caption{Graphs illustrating that each part of Conjecture~\ref{6part-conj} is
best possible.  In each case except (e), the square of the line graph is a
clique with order equal to the conjectured upper bound.  In contrast, (e) is the
Heawood graph $H$ with 21 edges.  Now no color can be used on more than 3 edges, so
$\chi^s(H)\ge 21/3=7$. 
\label{subcubic-sharpness-figs}}
\end{figure}

Wu and Lin~\cite{WuL08} proved part (d) in a stronger form.  Rather than
requiring the graph to be bipartite, they only forbid a single graph $H$, formed
from a 5-cycle by adding a path of length 2 joining two non-adjacent vertices.
The disproof of part (f) more generally characterizes~\cite{LMSS} the
$k$-regular graphs $G$ with $\chi^s(G)=2k-1$; these are graphs which cover the
Kneser graph $K(2k-1,k-1)$.  So, to disprove part (f), the authors construct
cubic bipartite graphs of arbitrarily large girth that have no such cover.
An alternative disproof of (f) is given in~\cite{cranston-FSGT-ce}.

We suspect that the bounds in the first 5 parts of Conjecture~\ref{6part-conj} also
hold in the context of list coloring (and possibly even paintability), but we are
unaware of many results in this direction.
One notable exception is recent work~\cite{LMSS26} of 
Lu\v{z}ar, M\'{a}\v{c}ajov\'{a}c
Sot\'{a}k, and \v{S}vecov\'{a}.
Relying largely on the Combinatorial Nullstellensatz,
they proved that every subcubic graph $G$ satisfies $\chi^s_{\ell}(G)\le 10$.

It is natural to ask: When do we have $\chil^s(G)=\chi^s(G)$?
In fact, \cite{DWYY} asked whether this is true for all graphs.
This was answered negatively by Lu\v{z}ar et al.~\cite{LMSS26}.
(Independently, Hasanvand~\cite[Theorem~3.5]{hasanvand} found list assignments showing
that $\chil^s(P(5k,2))>5$, where $P(n,2)$ denotes the $2n$-vertex generalized
Petersen graph.  It is known that $\chi^s(P(5k,2))=5$, so this provides another
infinite class of graphs witnessing a negative answer to the above question.)
Lu\v{z}ar et al. showed that $\chi^s(P)=5$, but $\chil^s(P)=7$, where $P$ is the Petersen graph.
Proving the first inequality is straightforward, but proving the second is harder.
However, their proof~\cite[Theorem~1.7]{LMSS26} that $\chil^s(P)>5$ is short and elegant.
More generally, building on the characterization~\cite{LMSS} of $k$-regular graphs $G$ with 
$\chi^s(G)=2k-1$, they conjectured that every such $G$ satisfies $\chil^s(G)>2k-1$.

Many of the parts of Conjecture~\ref{6part-conj} are now believed to hold in 
stronger forms.
\begin{conjecture} Let $G$ be a subcubic graph.
\label{subcubic-stronger-conj}
\begin{itemize}
\item[(a)]\cite{HLL} If $G$ is bridgeless and is neither the Wagner graph nor
the graph formed from the complete bipartite graph $K_{3,3}$ by subdividing one
edge, then $\chi^s(G)\le 9$.
\item[(b)]\cite{LMSS} If $G$ is bridgeless and $|V(G)|\ge 13$, then $\chi^s(G)\le 8$.
\item[(c)]\cite{LMSS} If $G$ has girth at least 5, then $\chi^s(G)\le 7$.
\end{itemize}
\end{conjecture}

Note that Conjecture~\ref{subcubic-stronger-conj}(a)
strengthens all of Conjecture~\ref{6part-conj}(a,b,c).
Conjecture~\ref{subcubic-stronger-conj}(b) even further strengthens
Conjecture~\ref{subcubic-stronger-conj}(a).
And Conjecture~\ref{subcubic-stronger-conj}(c) strengthens
Conjecture~\ref{6part-conj}(e), which still remains open.
The authors of~\cite{LMSS} also ask whether there exists a girth $g_0$ such that
every cubic graph $G$ with girth at least $g_0$ satisfies $\chi^s(G)\le 6$.

In the next section, we discuss bounds on $\chi^s$ for graphs with bounded
maximum average degree.  However, when $\Delta\in\{3,4\}$, the results and
proofs tend to be different from the general case.  So we include such work
here.

Assume that $\Delta=3$.  Hocquard and Valicov~\cite{HocquardV11} 
showed that if
$\mad(G)<\frac{36}{13}$ (resp. $\frac{13}5$, $\frac{27}{11}$,
$\frac{15}7$), then $\chi^s(G)\le 9$ (resp. 8, 7, and 6).  A few years later,
Hocquard, Montassier, Raspaud, and Valicov~\cite{HocquardMRV13} weakened the
hypotheses on $\mad(G)$.
They showed that if $\mad(G)<\frac{20}{7}$ (resp. $\frac{8}3$, $\frac52$,
$\frac{7}3$), then $\chi^s(G)\le 9$ (resp. 8, 7, and 6).  For all but the bound
implying $\chi^s(G)\le 8$, they gave constructions (each with at most eight vertices)
showing that the bound on $\mad(G)$ is sharp.  This problem has also 
been studied for list coloring.  Ma, Miao, Zhu, Zhang, and Luo~\cite{MaMZZL13}
showed that if $\mad(G)<\frac{36}{13}$ (resp. $\frac{13}5$, $\frac{27}{11}$,
$\frac{15}7$), then $\chil^s(G)\le 9$ (resp. 8, 7, 6).  
These bounds were improved by Zhu and Miao~\cite{ZhuM14}, who
showed that if $\mad(G)<\frac{14}{5}$ (resp. $\frac{8}3$, $\frac{5}{2}$), then
$\chil^s(G)\le 9$ (resp. 8, 7).  

Now assume that $\Delta=4$.  Bensmail, Bonamy, and Hocquard~\cite{BBH15} showed
that if $\mad(G)<\frac{19}5$ (resp. $\frac{18}5$, $\frac{17}5$, $\frac{10}3$,
$\frac{16}5$), then $\chi^s(G)\le 20$ (resp.~19, 18, 17, and 16).
Lv, Li, and Yu~\cite{LLY16} weakened these hypotheses, showing that if
$\mad(G)<\frac{51}{13}$
(resp. $\frac{15}4$, $\frac{18}5$, $\frac{7}2$, $\frac{61}{18}$), then
$\chi^s(G)\le 20$ (resp.~19, 18, 17, and 16).

\subsection{Planar Graphs and Bounded Maximum Average Degree}
\label{sec-chis-planar-girth}

Since planar graphs are sparse, we expect that they should satisfy a stronger
upper bound on $\chi^s$.  Indeed, Faudree, Schelp, Gy\'{a}rf\'{a}s, and
Tuza~\cite{FaudreeSGT90} combined Vizing's Theorem
and the Four Color Theorem to prove, for every planar graph $G$, that $\chi^s(G)\le 4\Delta+4$.
They also gave the following construction; see Figure~\ref{sharpness-figs}(a).
Take two copies of $K_{2,\Delta-2}$ and identify the
vertices of a 4-cycle in one copy with the vertices of a 4-cycle in the other
(so the resulting graph $G$ has maximum degree $\Delta$).  Note %It is easy to check
that $G$ is planar, that $|E(G)|=4\Delta-4$, and that a strong edge-coloring of
$G$ must give all edges distinct colors.  Thus, $\chi^s(G)=4\Delta-4$.
So, the theorem below is nearly sharp.

\begin{theorem} [\cite{FaudreeSGT90}]
If $G$ is a planar graph, then $\chi^s(G)\le4\Delta+4$.   If also $\Delta\ge 7$,
then $\chi^s(G)\le4\Delta$.
\label{FSGT90-planar}
\end{theorem}

\begin{proof}
By Vizing's Theorem, $G$ has an edge-coloring with at most $\Delta+1$ colors;
call these $c_1,\ldots,c_{\Delta+1}$.  For each color $c_i$, form $G_i$ from $G$
by contracting each edge colored $c_i$.  Since $G_i$ is planar, its vertices
have a 4-coloring; call its colors $b_{i,1},b_{i,2},b_{i,3},b_{i,4}$.  Since the
4-coloring is proper, each pair of edges colored $b_{i,j}$ (for some choice of
$i$ and $j$) must be distance at least two apart in $G$.  Thus, to get a strong
edge-coloring of $G$, we color each edge with its color $b_{i,j}$.
This uses at most $4(\Delta+1)$ colors.
If $\Delta\ge 7$, then $G$ is class 1, i.e., $G$ has an edge-coloring with only
$\Delta$ colors~\cite{SandersZhao-Delta7,Vizing65}; see also
\cite[Theorem 5.3]{CranstonW-guide}.
Thus, the proof above now yields $\chi^s(G)\le 4\Delta$.
\end{proof}

\begin{figure}[!t]
\centering
\begin{tikzpicture}[scale=.6, thick]
\tikzstyle{uStyle}=[shape = circle, minimum size = 5.5pt, inner sep = 0pt,
outer sep = 0pt, draw, fill=white, semithick]
\tikzset{every node/.style=uStyle}

\begin{scope}[yscale=.85,xshift=.25in]
\foreach \x/\y/\lab in {0/0/A, 1/2/B, 1/-2/C, 2/0/D, 6/0/E}
\draw (\x,\y) node (\lab) {};
\draw (B) -- (A) -- (C) -- (D) -- (B) -- (E) -- (C);

\foreach \height in {1,-1}
\draw (A) -- (1,\height) node {} -- (D);

\foreach \width in {3,5}
\draw (B) -- (\width,0) node {} -- (C);

\draw (1,.15) node[draw=none] {$\vdots$};
\draw (3.9,0) node[draw=none, shape=rectangle] {$\cdots$};
\draw (3,-4.10cm) node[draw=none] {\footnotesize{(a)}};
\end{scope}

\begin{scope}[xshift=4.5in]
\def\rad{.85}
\def\orad{1.25cm}
\draw (-\rad,-\rad) node (A) {} -- (-\rad,\rad) node (B) {} -- (\rad,\rad) node
(C) {} -- (\rad,-\rad) node (D) {} -- (A) -- (C) (D) -- (B);

\draw (A) ++ (190:\orad) node {} -- (A) --++ (260:\orad) node {} 
(228.5:2.8cm) node[draw=none, fill=none, label={[rotate=-45]\small{$\dots$}}] {};
\draw (B) ++ (100:\orad) node {} -- (B) --++ (170:\orad) node {}
(137:1.75cm) node[draw=none, fill=none, label={[rotate=45]\small{$\dots$}}] {};
\draw (C) ++ (80:\orad) node {} -- (C) --++ (10:\orad) node {}
(40:1.8cm) node[draw=none, fill=none, label={[rotate=-45]\small{$\dots$}}] {};
\draw (D) ++ (350:\orad) node {} -- (D) --++ (280:\orad) node {}
(313:2.75cm) node[draw=none, fill=none, label={[rotate=45]\small{$\dots$}}] {};
\draw (0,-3.5cm) node[draw=none] {\footnotesize{(b)}};
\end{scope}

\begin{scope}[xshift=7.5in, yshift=-.5cm]
\def\rad{1.2cm}
\def\orad{1.3cm}
\draw (90:\rad) node (A) {} -- (330:\rad) node (B) {} -- (210:\rad) node (C) {} -- (A);

\draw (A) ++ (60:\orad) node {} -- (A) --++ (120:\orad) node {}
(90:1.8cm) node[draw=none, fill=none, label={[rotate=0]\small{$\ldots$}}] {};
\draw (B) ++ (360:\orad) node {} -- (B) --++ (300:\orad) node {}
(327:2.825cm) node[draw=none, fill=none, label={[rotate=60]\small{$\ldots$}}] {};
\draw (C) ++ (180:\orad) node {} -- (C) --++ (240:\orad) node {}
(213:2.825cm) node[draw=none, fill=none, label={[rotate=-60]\small{$\ldots$}}] {};
\draw (0,-3cm) node[draw=none] {\footnotesize{(c)}};
\end{scope}

\end{tikzpicture}
\caption{Figures (a), (b), and (c) illustrate, respectively, that
Theorem~\ref{FSGT90-planar} is nearly sharp, that Theorem~\ref{CKKR-thm}(i) is
sharp, and that Theorem~\ref{CKKR-thm}(iii) is sharp.\label{sharpness-figs}}
\end{figure}

When $G$ has girth at least 7, we can use Gr\"{o}tzsch's Theorem to improve the
bound~\cite{HudakLSS} in Theorem~\ref{FSGT90-planar} to $\chi^s(G)\le 3\Delta$. 
When also $\Delta\ge 4$, Ruksasakchai and Wang~\cite{RW} strengthened this to
$\chil^s(G)\le 3\Delta$.

Hud\'{a}k, Lu\v{z}ar, Sot\'{a}k, and \v{S}krekovski~\cite{HudakLSS} showed that if
$G$ is planar with girth at least 6, then $\chi^s(G)\le 3\Delta+5$.  
Bensmail, Harutyunyan, Hocquard, and Valicov~\cite{BensmailHHV14} strengthened
this bound to $\chi^s(G)\le 3\Delta+1$.
The current strongest bounds in this direction are due to Choi, Kim, Kostochka and
Raspaud~\cite{CKKR} and Li, Li, Lv, and Wang~\cite{LLLW23}.
\begin{theorem}
\label{CKKR-thm}
(i)\cite{CKKR} If $\mad(G)<3$ and $\Delta\ge 7$, then $\chi^s(G)\le 3\Delta$.  
(ii)\cite{LLLW23} If $\mad(G)<\frac{26}9$ and $\Delta\ge 7$, then $\chi^s(G)\le 3\Delta-1$.
(iii)\cite{CKKR} If $\mad(G)<\frac83$ and $\Delta\ge 9$, then $\chi^s(G)\le 3\Delta-3$.
\end{theorem}

All planar graphs of
girth at least 6 have $\mad<3$, so Theorem~\ref{CKKR-thm}(i) extends the result
above of Bensmail et al. when $\Delta\ge 7$ and strengthens the bound by 1.  The
hypothesis $\mad(G)<3$ is sharp, as follows.  Given integers $t$ and $\Delta$,
form $G_{t,\Delta}$ from $K_t$ by adding $\Delta-t+1$ pendant edges at each
vertex; see Figure~\ref{sharpness-figs}(b,c).  Note that
$\mad(G_{4,\Delta})=3$ and $\chi^s(G_{4,\Delta})=4\Delta-6$.  In a similar
vein, the above bound of $3\Delta-3$ in (iii) is sharp, since
$\chi^s(G_{3,\Delta})=3\Delta-3$ and $\mad(G_{3,\Delta})=2$.

Cames van Batenburg, Kang, de Joannis de Verclos, and Pirot~\cite{vBKdJdVP} 
observed that the proof of Theorem~\ref{FSGT90-planar} 
%and the construction preceding it, 
can be further generalized.  For a graph class $\G$, let $\chi(\G)$ and 
$\omega(\G)$ denote the maximum of $\chi(G)$ and of $\omega(G)$ over all $G\in \G$, 
and let $\chi^s(\G,{\Delta})$ denote the maximum of $\chi^s(G)$ over all $G\in \G$ 
with maximum degree $\Delta$.  They showed that if $\G$ is minor-closed, then 
$\chi^s(\G)\le \chi(\G)(\Delta+1)$.
And if $G$ is also closed under adding pendant edges, then $\chi^s(G,\Delta) \ge 
\omega(\G)\Delta-\binom{\omega(G)}2$ when $\Delta > \omega(G)$.
The proof of the upper bound follows the proof of Theorem~\ref{FSGT90-planar}.
And the lower bound follows from $G_{\omega(\G),\Delta}$.
%\smallskip
\bigskip

\renewcommand{\arraystretch}{1.4}
\begin{figure}[!t]
\centering

\begin{tabular}{r|cccc}
girth $\ge$ & 3 & 3 & 7 & $10\Delta-4$ \\
\hline
$\Delta\ge$ & 1 & 7 & 1 & 4 \\
$\chi^s\le$ & $4\Delta+4$ & $4\Delta$ & $3\Delta$ & $2\Delta-1$
\end{tabular}
~~~~~~~~~~~~
\begin{tabular}{r|cccc}
$\mad<$ & 3 & 26/9 & 8/3  \\
\hline
$\Delta\ge$ & 7 & 7 & 9 \\
$\chi^s\le$ & $3\Delta$ & $3\Delta-1$ & $3\Delta-3$ 
\end{tabular}
\captionsetup{width=.9\linewidth}
\caption{Left: Triples $(g,D,f(\Delta))$, in columns, such that every planar graph $G$ with 
girth at least $g$ and $\Delta(G)\ge D$ 
satisfies $\chi^s(G)\le f(\Delta)$. 
Right: Triples $(k,D,f(\Delta))$ such that every graph $G$ with 
$\mad(G)<k$ and $\Delta(G)\ge D$ 
satisfies $\chi^s(G)\le f(\Delta)$. 
References for both tables appear throughout Section~\ref{sec-chis-planar-girth}.}
%graph with maximum degree $\Delta=3$ and girth at least $g_k$.  (The entry $(3,-)$
%denotes that no such pair $(3,g_3)$ exists. References are in the preceeding few
%paragraphs.)}
\end{figure}

Next we return to planar graphs with sufficiently large girth.  In a sense,
such graphs feel ``tree-like'', so perhaps a trivial lower bound may hold with
equality, as it does for trees.  Note that if a graph $G$ has adjacent vertices
of degree $\Delta$, then $\chi^s(G)\ge 2\Delta-1$.  So we seek sufficient
conditions on planar graphs to imply $\chi^s(G)\le2\Delta-1$.
Borodin and Ivanova~\cite{BorodinI13StrongEdge} showed that 
$\Delta\ge 3$ with $g\ge40\floor{\Delta/2}+1$ suffices.
Chang, Montassier, P\^{e}cher, and Raspaud~\cite{ChangMPR14} proved that 
also $\Delta\ge 4$ with $g\ge 10\Delta+46$ suffices, which is a
stronger result when $\Delta\ge 6$.  
Most recently, Wang and Zhao~\cite{WZ18} weakened the hypotheses further: $g\ge 10\Delta-4$;
their result is the best known when $\Delta\ge 4$.

\begin{theorem}[\cite{WZ18}]
If $G$ is planar with $\Delta\ge 4$ and $g\ge 10\Delta-4$, then $\chi^s(G)\le
2\DeltaG-1$.
\end{theorem}

Hud\'{a}k, Lu\v{z}ar, Sot\'{a}k, and S\v{k}rekovski~\cite{HudakLSS} continued the
study of $\chi^s$ for planar graphs in terms of both their girth $g$
and their maximum degree $\Delta$.  But now they removed the dependence of $g$ on
$\Delta$.  The following construction provides a lower bound.
Given integers $g$ and $\Delta$, form $H_{g,\Delta}$ from the cycle $C_g$ by
adding $\Delta-2$ pendant edges at each cycle vertex.  Now
$|E(H_{g,\Delta})|=g(\Delta-2+1)$.  When $g$ is odd, the maximum size of an
induced matching is $(g-1)/2$.  So $\chi^s(H_{g,\Delta})\ge \lceil
2g(\Delta-1)/(g-1)\rceil$.  They conjectured that this construction is extremal,
up to an additive constant.

\begin{conjecture}[\cite{HudakLSS}]
\label{HLSS-conj}
There exists a constant $C$ such that if $G$ is any planar graph with girth
$g\ge 5$ and maximum degree $\Delta$, then
$$\chi^s(G)\le \left\lceil\frac{2g(\Delta-1)}{g-1}\right\rceil+C.$$
\end{conjecture}

The best progress on Conjecture~\ref{HLSS-conj} is due to Chen, Deng, Yu, and
Zhou~\cite{CDYZ}, who proved the following.
Let $G$ be connected and planar with girth $g$, where $g\ge 26$.
If $G$ is not a subgraph of $H_{g,\Delta}$ (in the previous paragraph), then
$\chi^s(G)\le 2\Delta+\lceil 12(\Delta-2)/g\rceil$.

\section{Odds and Ends}
In this section we briefly touch on a few more related problems. 
%\subsection{Total Coloring}
\subsection{Total Coloring and
List Coloring Squares} %, and List Coloring Total Graphs}
\label{sec:total}
A \Emph{total coloring} of a graph $G$ colors its edges and vertices so that
every two adjacent or incident elements get distinct colors.  The \emph{total chromatic
number} $\chi''(G)$\aaside{$\chi''$}{3mm}, is the minimum number of colors needed for a total coloring.
The \EmphE{total graph}{3mm} $T(G)$ of a graph $G$ has as its vertices the edges and
vertices of $G$; two vertices of $T(G)$ are adjacent if their corresponding
elements are incident or adjacent in $G$, so $\chi''(G)=\chi(T(G))$. 
Equivalently, $T(G)$ is the square of its edge-vertex incidence
graph; this incidence graph is formed from $G$ by subdividing each
edge. Thus, total coloring is a very special case of
coloring squares.  Undoubtedly, the biggest conjecture in this area is
the following ``Total Coloring Conjecture'' that was posed independently by Bezhad~\cite{Behzad} and
Vizing~\cite{Vizing64, Vizing65b}.

\begin{conjecture}[\cite{Behzad,Vizing64,Vizing65b}]
Every graph $G$ satisfies $\chi''(G)\le\Delta+2$.
\label{conj:total}
\end{conjecture}

It is easy to check that the conjecture holds for every clique, and that it holds
with equality for every even clique, as follows.  
For the upper bound, let $G:=K_{2t+1}$.  By Vizing's Theorem, $\chi'(G)\le
\Delta+1=2t+1$.  Counting edges implies that each of the $2t+1$ color classes
must be a matching of size $t$; further each vertex has an incident edge in all
but one matching.  Thus, we can extend the edge-coloring to a total coloring,
with no new colors.  So $\chi''(K_{2t+1})=2t+1=\Delta+1$.  Hence,
$\chi''(K_{2t})\le 2t+1=\Delta+2$.
Now, we give the matching lower bound for $K_{2t}$.  We must color $\binom{2t}
{2}+2t=2t^2+t$ elements.  Since each color appears on at most $t$ elements, we
need at least $(2t^2+t)/t = 2t+1=\Delta+2$ colors.  Hence, $\chi''(K_{2t})=\Delta+2$.

By applying Brooks' Theorem to the total graph, we can see that $\chi''(G)\le
2\Delta$.  One approach to improving this bound is to color the vertices first,
then the edges.  The vertices need at most $\Delta+1$ colors, and each edge
loses colors to only its two endpoints.  So this idea shows that $\chi''(G)\le
2+\chil'(G)$.  
Unfortunately, it is non-trivial to bound $\chil'(G)$ below
$2\Delta-2$  (although this can be done~\cite{Kahn,Kahn2}).

The first bound of the form $\chi''(G)\le \Delta+o(\Delta)$ is due to
Hind~\cite{Hind-total}, who showed that $\chi''(G)\le \Delta+2\sqrt{\Delta}$.
Chetwynd and H\"{a}ggkvist~\cite{ChetwyndH}
improved this to $\chi''(G)\le\Delta+18\Delta^{1/3}\log(3\Delta)$.
And for $\Delta$ sufficiently large, Hind, Molloy, and Reed~\cite{HindMR}  
strengthened the bound to $\chi''(G)\le\Delta+\textrm{poly}(\log\Delta)$.
%Brualdi~\cite[p.~437]{brualdi-ref} and 
Alon~\cite{Alon93} asked whether there is
some constant $C$ such that $\chi''(G)\le\Delta+C$ for every graph $G$.
The best current bound, due to Molloy and
Reed~\cite{MolloyR-total} and \cite[Chapters 17--18]{MolloyR-GCPM},
answers their question affirmatively.

\begin{theorem}[\cite{MolloyR-total}]
There exists $\Delta_0$ such that if $G$ has $\Delta\ge \Delta_0$,
 then $\chi''(G)\le \Delta+10^{26}$.
\label{thm:MolloyR-total}
\end{theorem}

In the paper, the authors note that by being more careful, they can show that
$C=500$ suffices and probably even $C=100$ (but most likely not $C=10$).
It seems that no analogue of Theorem~\ref{thm:MolloyR-total} is known for
list coloring.  However, Kahn's results mentioned above for list coloring line
graphs~\cite{Kahn,Kahn2} do imply that $\chil''(G)\le\Delta+o(\Delta)$.

Thus far in this survey, we have avoided fractional coloring.  Here we make a
brief exception.  For an introduction to this topic,
see~\cite{ScheinermanU-book}.  To denote the fractional total chromatic number,
we write $\chif''(G)$.  Kilakos and Reed~\cite{KilakosR} proved the relaxation of
Conjecture~\ref{conj:total} for fractional coloring.  Specifically, they showed
that $\chif''(G)\le \Delta+2$ for every graph $G$ (see
also~\cite[Section 4.6]{ScheinermanU-book}).  Reed conjectured that for
every $\epsilon>0$ and every $\Delta$ there exists some girth $g$ such that if
$G$ is a graph with girth at least $g$ and maximum degree $\Delta$, then
$\chif''(G)\le \Delta+1+\epsilon$.  For $\Delta=3$ and for $\Delta$ even, Kaiser,
King, and \kral~\cite{KaiserKK} proved the conjecture in a stronger sense.
For each such $\Delta$, they showed there exists $g$ such that if $G$ has
maximum degree $\Delta$ and girth at least $g$, then $\chif''(G)=\Delta+1$.
Kardo\v{s}, \kral, and Sereni~\cite{KardosKS} verified Reed's conjecture for
the remaining case, odd $\Delta$ (although not in the stronger sense mentioned above).

Now we consider some graph classes for which Conjecture~\ref{conj:total} is proved.
For bipartite graphs, the result holds trivially, since $\chi''(G)\le
\chi(G)+\chi'(G)=2+\Delta$.  In fact, $\chip''(G)\le \Delta+2$, but the proof is
harder.  Galvin~\cite{galvin} showed that $\chil'(G)=\Delta$ for every bipartite
graph, and Schauz~\cite{Schauz09} strengthened this to $\chip'(G)=\Delta$.  By
coloring the vertices first, we get $\chil''(G)\le 2+\chil'(G)=\Delta+2$; a
similar argument shows that also $\chip''(G)\le\Delta+2$.
So now suppose $G$ is non-bipartite.
For $\Delta=2$, the problem reduces to computing $\chi^2$ for odd cycles, so
$\chiAT''\le \Delta+2=4$, as noted at the end of Section~\ref{sec:planar-girth}.
For $\Delta=3$, the conjecture was proved by
Rosenfeld~\cite{Rosenfeld} and by Vijayaditya~\cite{Vijayaditya}.  Kostochka proved
it for $\Delta=4$~\cite{Kostochka-total4} and also for
$\Delta=5$, in his dissertation; see~\cite{Kostochka-total}.  
Next, we consider~planar~graphs.

\begin{theorem}[\cite{Andersen93}]
Every planar graph $G$ with $\Delta\ge 7$ satisfies $\chi''(G)\le \Delta+2$.
\label{planar-total}
\end{theorem}
\begin{proof}
By the Four Color Theorem, we properly color the vertices with colors
$\{1,2,3,4\}$.  Since $G$ is planar and $\Delta\ge 7$, we can properly color
the edges with $\Delta$ colors~\cite{SandersZhao-Delta7,Zhang-Delta7}; we use
colors $\{3,4,\ldots,\Delta+1,\Delta+2\}$.  The only possible conflicts involve
edges colored with 3 or 4, so we uncolor all such edges.  
Now each uncolored edge has available two of the colors $\{1,2,3,4\}$, since
only two of these are used on its endpoints.  
Note that this uncolored subset of edges induces a vertex disjoint union of
paths and even cycles (since the edges were properly colored with 3 and 4).  
Thus, we can color the edges from their lists of available colors, since 
paths and even cycles have list chromatic number 2.
\end{proof}

\noindent
The idea of this proof is due to Yap~\cite[p.~88]{JensenToft}, but the first
presentation with all the details seems to be due to Andersen~\cite{Andersen93}.

For planar graphs with $\Delta$ sufficiently large, the bound in
Theorem~\ref{planar-total} can be improved
to $\chi''(G)= \Delta+1$.  Borodin~\cite{Borodin89} showed that $\Delta\ge 14$
suffices.  This hypothesis was successively weakened to $\Delta\ge 11$ by
Borodin, Kostochka, and Woodall~\cite{BorodinKW97b}; to $\Delta\ge 10$ by
Wang~\cite{Wang07}; and to $\Delta\ge 9$ by Kowalik, 
Sereni, and \v{S}krekovski~\cite{KowalikSS08}.

When $\Delta\le 8$, it is natural to look for additional hypotheses that imply
$\chi''(G)=\Delta+1$.  Chen and Wu~\cite{ChenW93}
showed that $\chi''(G)=\Delta+1$ for planar
graphs with  $\Delta\ge D$ and $g\ge k$ whenever
$(D,k)\in\{(8,4),(6,5),(4,8)\}$.
Borodin, Kostochka, and Woodall~\cite{BorodinKW98} strengthened these results
by weakening the hypotheses to 
$(D,k)\in\{(7,4),(5,5),(4,6),(3,10)\}$.
For planar graphs with $\Delta\in\{7,8\}$, the bound $\chi''(G)=\Delta+1$ has
also been proved when various types of cycles are forbidden; see, for example, \cite{
total-partial4, 
total-partial6,
total-partial3, 
total-partial2, 
total-partial5,
total-partial1}.  
For the case $\Delta=6$, the best result
is $\chi''(G)\le 9$.  This was originally proved by Borodin~\cite{Borodin89}, 
but also follows immediately from the fact that $\chi''(G)\le \Delta+2$ when
$\Delta=7$.

%\subsection{List Coloring Squares and Total Graphs}
\label{sec:LCS}
The most famous conjecture in list coloring states that every line graph
$G$ has $\chil(G)=\chi(G)$. 
In 2001, Kostochka and Woodall~\cite{KostochkaW} posed
an analogous conjecture for squares.

\begin{SLCC}[\cite{KostochkaW}]
For every graph $G$, we have $\chil^2(G)=\chi^2(G)$.
\end{SLCC}

Every graph $G$ clearly satisfies $\chil^2(G)\ge\chi^2(G)\ge \Delta+1$. 
However, proving a better lower bound on $\chi^2(G)$ or even $\chil^2(G)$ is
typically quite
difficult.  Thus, the majority of graphs $G$ known to satisfy the Square List
Coloring Conjecture are those for which $\chil^2(G)=\chi^2(G)=\Delta+1$.
This conjecture was proved for many classes of graphs~\cite{BLP1, BLP2,
BIN07, CranstonS, DvorakST08}, but in general it is false, as shown in 2013 by
Kim and Park~\cite{KimP-SLCC}.

\begin{theorem}[\cite{KimP-SLCC}]
\label{KimP-disproof}
The Square List Coloring Conjecture is false. More specifically, for each
integer $k$, there exists a graph $G_k$ such that $\chil(G^2_k)-\chi(G^2_k)>k$.  
\end{theorem}
\begin{proof}[Proof Sketch.]
A \emph{latin square} is an arrangement of $k$ copies of each of the symbols
$1,\ldots,k$ in cells of a $k\times k$ grid, so that each symbol appears
exactly once in each row and
each column.  When we overlay one latin square with another, we form in each
cell of the grid an ordered pair, $(x,y)$, with $x$ coming from the first
square and $y$ from the second.  Two latin squares are \emph{mutually orthogonal} 
if each of the possible ordered pairs appears in exactly one cell of the grid.
It is well known that for every prime $k$, there exists a family of $k-1$
pairwise mutually orthogonal latin squares.

Kim and Park used these mutually orthogonal latin squares to construct a graph
$H_k$ such that $H_k^2=K_{k*(2k-1)}$, the complete multipartite graph with
$2k-1$ parts, each of size $k$.  Clearly, this gives $\chi(H_k^2)=2k-1$.
Further, it is known~\cite{vetrik} that
$\chil(K_{k*r})>(k-1)\floor{\frac{2r-1}k}$.  In particular, this gives
$\chil(K_{k*(2k-1)})>3(k-1)$.  Now taking $G_k=H_{k+2}$ gives
$\chil(G_k^2)-\chi(G_k^2)> k$, as desired.
\end{proof}
Later Kim and Park~\cite{KimP2} showed that the graphs $G_k$ in
Theorem~\ref{KimP-disproof} can also be required to be bipartite.
In the direction of the Square List Coloring Conjecture, Zhu asked whether there
exists a constant $K$ such that $\chil^k(G)=\chi^k(G)$ for all $k\ge K$ and all
graphs $G$.  This question was answered negatively by Kosar, Petrickova,
Reiniger, and Yeager~\cite{KPRY} and also by Kim, Kwon, and Park~\cite{KKP}.
The authors of~\cite{HavetHMR} attempted to partially salvage the Square List
Coloring Conjecture, proposing that it holds for all planar graphs.  However,
this too was disproved.  See Conjecture~\ref{HvdHMR-conj} and
Theorem~\ref{HvdHMR-c/e}.

Recall from above %Section~\ref{sec:total} 
that $\chi''(G)=\chi(T(G))$, 
where the \emph{total graph} $T(G)$ of a graph $G$ has as its vertices the edges and
vertices of $G$; two vertices of $T(G)$ are adjacent if their corresponding
elements are incident or adjacent in $G$.  Equivalently, the total graph of $G$
is formed from $G$ by subdividing each edge of $G$, then taking the square. 
The disproof of the full Square List Coloring Conjecture
(Theorem~\ref{KimP-disproof} above)
has increased interest in the following special case, which was posed
earlier by Borodin, Kostochka, and Woodall~\cite{BorodinKW97}.

\begin{TLCC}[\cite{BorodinKW97}]
For every graph $G$, if $T(G)$ is the total graph of $G$, then
$\chil(T(G))=\chi(T(G))$.
\end{TLCC}

In~\cite{BorodinKW97} the authors proved the conjecture for every
simple planar graph with $\Delta\ge 12$.  The same paper included a
clever averaging argument, which proved the conjecture whenever $\mad(G)\le
\sqrt{2\Delta}$.  The proof of the latter result was significantly simplified by
Woodall~\cite{woodall-ave-deg}; see also~\cite[Section 4]{CranstonW-guide}.
A \emph{multicircuit} is a multigraph for which the underlying simple graph is a
cycle.  Kostochka and Woodall proved the Total List Coloring Conjecture for all
multicircuits~\cite{multicircuits1,multicircuits2}; however, in general, it
remains open.

\bigskip
\subsection{L(2,1)-Labeling}
\label{sec:L21}

An \emph{$L(p,q)$-labeling} of a graph $G$ assigns to each vertex of $G$ a
positive integer such that each pair of vertices at distance 1\ in $G$ receives
integers differing by at least $p$ and each pair of vertices at distance 2\ in $G$
receives integers differing by at least $q$.  The \emph{span} of such a
labeling is the difference between its largest and smallest integers. The
minimum span over all $L(p,q)$-labelings of a graph $G$ is
$\lambda^{p,q}(G)$.\aside{$\lambda^{p,q}$}  When $p=q=1$, the problem is
equivalent to coloring $G^2$; however, note that $\chi^2(G) = \lambda^{1,1}(G)
+ 1$.  The next most widely studied case is when $p=2$ and $q=1$.  
Griggs and Yeh~\cite{GriggsY} posed the following intriguing conjecture.

\begin{L21}[\cite{GriggsY}]
Every graph $G$ has $\lambda^{2,1}(G)\le \Delta^2$.
\label{conj:L21}
\end{L21}

Consider a greedy $L(2,1)$-labeling of $G$ in an arbitrary order.
For each vertex $v$, each of at most $\Delta$ neighbors of $v$ forbids at most
three labels on $v$.
Similarly, each of at most $\Delta(\Delta-1)$ vertices at distance 2\ from $v$
forbids at most one label on $v$.  Thus, a greedy $L(2,1)$-labeling uses no label
larger than $1+3\Delta+\Delta(\Delta-1)$.  So $\lambda^{2,1}(G)\le
\Delta^2+2\Delta$.  Chang and Kuo~\cite{ChangK} gave the first major
improvement of this bound, showing that $\lambda^{2,1}\le \Delta^2+\Delta$, for
every graph.  This upper bound was further strengthened to $\Delta^2+\Delta-1$
by {\kral} and \v{S}krekovski~\cite{KralS} and to $\Delta^2+\Delta-2$ by
Gon\c{c}alves~\cite{Goncalves}.
This result of Gon\c{c}alves remains the best bound in general, although for
$\Delta$ sufficiently large Havet, Reed, and Sereni 
have proved the $L(2,1)$-Labeling Conjecture.  In fact, they proved the same upper
bound~\cite{HavetRS12} for $L(p,1)$-labeling in general.

\begin{theorem}[\cite{HavetRS12}]
\label{HRS12-thm}
For each $p\ge 1$, if $G$ has $\Delta$ sufficiently large,
then $\lambda^{p,1}(G)\le\Delta^2$.
\end{theorem}

Numerous authors have written entire surveys on $L(p,q)$-labelings 
and their generalizations, such as real number labeling.  Thus, we direct the
interested reader to these~\cite{calamoneri-survey, GriggsK, yeh-survey}.
We close this short section with a sketch of the cute result~\cite{GriggsY} that
the $L(2,1)$-Labeling Conjecture holds for all graphs of diameter 2.
\begin{theorem}[\cite{GriggsY}]
If $G$ has diameter 2, then $\lambda^{2,1}(G)\le\Delta^2$.
\end{theorem}
\begin{proof}[Proof Sketch.]
Let $n:=|V(G)|$.  It is easy to check that $\lambda^{2,1}$ is at most 4 for paths and
cycles, so assume $\Delta\ge 3$.
First, suppose that $n\le 2\Delta+1$. If $\Delta\ge 4$, then label the vertices
arbitrarily with distinct elements of $\{0,2,4,\ldots,2(n-1)\}$.
This labeling is valid, and it has span at most
$2(n-1)\le2(2\Delta)\le\Delta^2$.
If $\Delta=3$, a slight modification works, labeling at most two vertices with
odd labels.
So assume instead that $n\ge 2\Delta+2$.
Let $d_{\overline{G}}(v)$ denote the degree of $v$ in $\overline{G}$, the
complement of $G$.  Since $n\ge 2\Delta+2$, we have
$d_{\overline{G}}(v)\ge n-(\Delta+1)\ge n/2$ for all $v$.  By Dirac's Theorem,
$\overline{G}$ contains a Hamiltonian path; call it $v_1, v_2,\ldots, v_n$.  Now label
$v_i$ with integer $i$.  Since $G$ is diameter 2, we have $n\le\Delta^2+1$, so
the span of this labeling is at most $\Delta^2$.  Since the labels are distinct,
we need only show, for each $i$, that $v_i$ and $v_{i+1}$ are non-adjacent.  This
is true because $v_1,v_2,\ldots, v_n$ is a Hamiltonian path in $\overline{G}$.
\end{proof}

This approach was extended by Cole~\cite{franks}, who showed that
$\lambda^{2,1}(G)\le\Delta^2$ whenever $G$ has order at most $(\lfloor
\Delta/2\rfloor+1)(\Delta^2-\Delta+1)-1$.

\subsection{Proper Conflict-free and Proper \texorpdfstring{$h$}{h}-Conflict-free Coloring}

A \emph{proper conflict-free coloring} of a graph $G$ is a proper coloring such that each
non-isolated vertex $v$ has a color $c_v$ that appears exactly once on $N(v)$.  We write
$\chi^{pcf}(G)$ to denote the minimum number of colors admitting such a coloring.   This topic
was introduced by Caro, Petrusevski, and Skrekovski, who conjecture that $\chi^{pcf}(G)\le \Delta(G)+1$
for every graph $G$.  It is also natural to pose an analogous conjecture for list-coloring:
$\chil^{pcf}(G)\le \Delta(G)+1$.  This problem has been widely~\cite{hickingbotham,liu,CCKP} studied.
The best general bounds are due to Liu and Reed~\cite{LR}: $\chi^{pcf}(G)\le \Delta(G)(1+o(1))$;
and to Cranston and Liu~\cite{CL}: $\chil^{pcf}\le \Delta(G)(1.656+o(1))$.  

But here we are more interested
in a recent generalization.  A proper $k$-coloring of a graph $G$ is \emph{$h$-conflict-free} if every
vertex $v$ has at least $\min\{h,d(v)\}$ colors that appear exactly once on $N(v)$.  Note that 
$h$-conflict-free colorings interpolate smoothly between proper colorings (0-conflict-free colorings)
and colorings of the square ($\Delta$-conflict-free-colorings).  Let $\chi^{h-pcf}(G)$ denote the 
minimum number of colors admitting such a coloring.
Chuet, Dai, Ouyang, and Pirot~\cite{CDOP} conjectured the following.

\begin{conjecture}[\cite{CDOP}]
\label{CDOP-conj}
For every positive integer $h$, there exists a constant
$D_h$ such that $\chi^{h-pcf}(G)\le h\Delta+1$ for all graphs $G$ with $\Delta(G)\ge D_h$.
\end{conjecture}

In the direction of Conjecture~\ref{CDOP-conj}, those authors proved that $\chi^{h-pcf}(G)\le h\Delta+O(h\log \Delta)$.  It is also natural
to extend this conjecture to list-coloring: For all graphs with $\Delta$ sufficiently large, we have 
$\chil^{h-pcf}(G)\le h\Delta+1$.  

\subsection{Injective Edge Coloring}
An injective coloring gives distinct colors to every pair of vertices with a common neighbor; but it
need not be proper.  So every proper coloring of $G^2$ is injective for $G$, but not vice versa.
For example, $\chi^2(C_5)=5$, but $\chi^{inj}(C_5)=3$.  For many classes $\G$ of sparse graphs,
the quantity $\max_{G\in \G}\chi^2(G)-\max_{G\in\G}\chi^{inj}(G)$ is known or conjectured to be quite
small.  Observe that always $\chi^2(G)\ge \Delta(G)+1$ and $\chi^{inj}(G) \ge \Delta(G)$.
And whenever $G$ is a tree, both lower bounds hold with equality.  Because $\chi^{inj}$ so often
mirrors $\chi^2$, we do not say much about the former in general.

But, for injective edge-coloring the situation seems to be different.  An edge-coloring $\vph$ of $G$
(not necessarily proper) is injective if $\vph(e_1)\ne \vph(e_2)$ for every pair $e_1,e_2$ of edges
that either are at distance exactly 1 in $G$ or lie in a triangle in $G$.  
The \emph{injective edge
chromatic number} $\chi'_{inj}(G)$ of $G$ is the minimum number of colors that admit an injective 
edge-coloring. 
This notion is less widely studied, but we discuss it here briefly to highlight some nice open problems.
(It is worth noting that in general $\chi'_{inj}(G)\ne \chi_{inj}(L(G))$.  For example, when $G$ is the
star $K_{1,s}$, with $s\ge 3$, we have $\chi'_{inj}(G) = 1$ but $\chi_{inj}(L(G))=s$.)

\begin{conjecture}[\cite{FKR}] 
\label{chi-inj-conj}
If $\Delta(G)=3$, then $\chi'_{inj}(G)\le 6$.  
And if $G$ is also bipartite, then $\chi'_{inj}(G)\le 5$.
\end{conjecture}

In support of the first statement, Kostochka, Raspaud, and Xu~\cite{KRX} showed that if $\Delta(G)=3$,
then $\chi'_{inj}(G)\le 7$.  And if $G$ is also planar, then $\chi'_{inj}(G)\le 6$.  A place where
we see a marked difference between $\chi^2(L(G))$ and $\chi'_{inj}(G)$ is for bipartite graphs.  
If $G$ is bipartite with maximum degree $\Delta$, then $\chi'_{inj}(G)\le \lceil 27\Delta\ln \Delta\rceil$.
More generally~\cite{KRX}, we have $\chi'_{inj}(L(G))\le (\chi(G)-1)\lceil 27\Delta\ln \Delta\rceil$.
(In contrast, $\chi^s(K_{\Delta_A,\Delta_B})=\Delta_A\Delta_B$.)  The authors of~\cite{KRX} note their 
proof of this bound does not work for the list-coloring analogue.  Thus, we pose the following question.

\begin{question}
\label{chilinj-ques}
Does there exist a constant $C$ such that for every bipartite graph $G$ with maximum degree $\Delta$ sufficiently large and each edge list assignment $L$ with $|L(e)|\ge C\Delta\ln \Delta$ 
for all $e\in E(G)$ we always have an injective $L$-edge-coloring?
\end{question}

More recent work in this area studied $d$-degenerate graphs and graphs embeddable in each fixed surface.
Bradshaw, Clow, and Xu~\cite{BCX} showed that if $G$ is $d$-degenerate, then 
$\chi'_{inj}(G)=O(d^3\ln \Delta)$; in this special case, they greatly improve the above bound of Kostochka et al.  They also showed that if $G$ is a graph of Euler genus $g$, then $\chi'_{inj}(G)\le (3+o(1))g$.

\subsection{Higher Powers}

Recall that the $k^{\textrm{th}}$
power, $G^k$, of a graph $G$ is formed from $G$ by adding an edge between each
pair of vertices at distance at most $k$ in $G$.  
%Various researchers have studied higher powers of graphs.
Let \Emph{$D_{k,\Delta}$} denote the largest possible degree
of a vertex in a $k^{\textrm{th}}$ power of a graph with maximum degree
$\Delta$.  It is easy to check that
$D_{k,\Delta}=\sum_{i=1}^k\Delta(\Delta-1)^{i-1} =
\Delta((\Delta-1)^k-1)/(\Delta-2)$.  
When $k\ge 3$ the situation is somewhat simpler than for $k=2$, since there does
not exist any graph $G_{k,\Delta}$ with maximum degree $\Delta$ (and $\Delta\ge
3$) such that $G_{k,\Delta}^k=K_{D_{k,\Delta}+1}$.  This was proved by
Damerell~\cite{damerell} and by Bannai and Ito~\cite{BI}.  (Both proofs
followed the general approach of Hoffman and Singleton for showing the
nonexistence of diameter 2 Moore graphs except when $\Delta\in\{2,3,7,57\}$:
studying the eigenvalues of a hypothetical such graph and reaching a
contradiction. Recall that this approach was outlined immediately preceeding
Lemma~\ref{lem:no-big-cliques}.) As a result, an analogue of
Theorem~\ref{CRpaint} for $k\ge 3$ does not need any exceptional graphs.
In fact, Bonamy and Bousquet~\cite{BonamyB} showed, for each $k\ge 3$ and graph
$G$ with maximum degree $\Delta$, that $\chil^k(G)\le D_{k,\Delta}-1$. 
And they conjectured something even stronger: For each integer $k\ge 3$, except
for a finite number of graphs $G$, 
every connected graph $G$ satisfies
$\chil^k(G)\le D_{k,\Delta}+1-k$.

The motivation for this conjecture follows our proof sketch of
Theorem~\ref{CRpaint}.  We greedily color the vertices in order of
non-increasing distance from some subgraph $H$ of diameter at least $k$.  Each vertex
$v$ outside $H$ has at least $k$ neighbors in $G^k$ that are closer to
$H$ (the first $k$ on a shortest path in $G$ from $v$ to $H$) and
are thus uncolored at the time we color $v$.  So the problem reduces to
proving that $G$ always contains a good subgraph $H$.  For coloring (but not list
coloring), their conjecture was confirmed by Pierron~\cite{pierron}.
\begin{theorem}[\cite{pierron}]
\label{pierron-thm}
Fix integers $k\ge 3$, $\Delta\ge 3$.  All but finitely many connected 
graphs $G$ with maximum degree $\Delta$ have $\chi^k(G)\le
D_{k,\Delta}+1-k$.  (Here $D_{k,\Delta}=\sum_{i=1}^k\Delta(\Delta-1)^{i-1}$.)
\end{theorem}
\noindent
The reasoning above suggests that perhaps also $\chip^k(G)\le D_{k,\Delta}+1-k$,
but we are unaware of any progress in this direction.

Now we turn to lower bounds.  
Let $n_{k,\Delta}$ denote the largest order
of a graph with maximum degree $\Delta$ and diameter $k$.
Bollob\'{a}s~\cite{bollobas_extremal-GT-book} conjectured, for every
$\epsilon>0$, that we have $n_{k,\Delta}>(1-\epsilon)\Delta^k$ for $\Delta$ and $k$
both sufficiently large.  It seems the best result in this direction is that
$n_{k,\Delta}\ge \left(\frac{\Delta}{1.6}\right)^k$ for all $k$ and an infinite
set of values of $\Delta$.  This was proved by Canale and G\'{o}mez~\cite{CG}.

It is also natural to consider coloring powers of graphs from some class, such
as planar graphs, chordal graphs, or line graphs.  In many cases the best known
bounds on $\chi^k(G)$ come from bounds on $\col^k(G)$, and most work gives only
asymptotic bounds.  

Agnarsson and Halld\'{o}rsson~\cite{AgnarssonH03} proved that if $G$ is planar,
then $\col^k(G)=O(\Delta^{\floor{k/2}})$. This is best possible, as shown by a
maximum tree with diameter $k$ and maximum degree $\Delta$.
\kral~\cite{Kral-chordal} showed that if $G$ is chordal, then
$\col^k(G) = O(\sqrt{k}\Delta^{(k+1)/2})$ when $k$ is even and
$\col^k(G)=O(\Delta^{(k+1)/2})$ when $k$ is odd. For odd $k$ this is again best
possible.  Now the construction is similar, but the root of the tree is replaced
by a clique on $\Delta/2$ vertices.  

For coloring powers of line graphs, greedy
coloring gives the easy bound $\chi^k(L(G))\le 2\Delta^k$.  Kaiser and
Kang~\cite{KaiserK} improved this to $\chi^k(L(G))\le(2-\epsilon)\Delta^k$ for
some $\epsilon>0$.
In a related question, \erdos~and \nesetril~\cite{erdos-nesetril} asked for the
minimum number of edges $h^k(\Delta)$ such that if any graph $G$ has maximum
degree at most $\Delta$ and at least $h^k(\Delta)$ edges, then its line graph
$L(G)$ has diameter at least $k+1$.  It is trivial to check that
$h_1(\Delta)=\Delta+1$.  Chung, Gy\'{a}rf\'{a}s, Tuza, and Trotter exactly determined
$h_2(\Delta)$; it is $\frac54\Delta^2+1$ when $\Delta$ is even and slightly
smaller when $\Delta$ is odd (see the start of
Section~\ref{erdos-nesetril-sec}).  For larger $k$, 
Cambie, Cames van Batenburg, de Joannis de Verclos, and Kang~\cite{CCdVK} 
showed that $\omega^k(L(G))\le \frac32\Delta^k$.  In particular,
$h^k(\Delta)\le \frac32\Delta^k+1$.  This implies that $\chi^k(L(G))\le
1.941\Delta^k$ for $\Delta$ sufficiently large (strengthening the result above
of Kaiser and Kang).

Let $f^k(\Delta,g)$ denote the maximum value of $\chi^k(G)$ over all
graphs $G$ with girth $g$ and maximum degree $\Delta$.  Alon and
Mohar~\cite{AlonM02} determined the
asymptotic value of $f^2(\Delta,g)$, when $g$ is fixed and $\Delta$ grows.
They showed that for $g\le 6$, we have $f^2(\Delta,g)=\Delta^2(1+o(1))$.  The
upper bound comes trivially from greedy coloring.  
For $k\ge 2$ and $g\ge 3k+1$, they showed that there exists a constant
$C_1$ such that $f^k(\Delta,g)\le C_1\Delta^k/\log \Delta$.  One way to prove 
this is using Johansson's result~\cite{johansson} for list coloring
triangle-free graphs; also see~\cite[Chapters 12--13]{MolloyR-GCPM}.  He showed
that there exists a constant $C_2$ such that
every triangle-free graph $G$ satisfies $\chil(G)\le C_2\Delta/(\log \Delta)$. 
Since $G^k$ has girth at least $\ceil{g/k}$ and maximum degree $O(\Delta^k)$,
the result follows.  To prove an asymptotically matching lower bound for
$f^k(\Delta,g)$ when $g\ge 2k+3$, Kaiser and Kang~\cite{KaiserK} gave a random
construction.  (Alon, Krivelevich, and Sudakov~\cite{AKS} extended this result,
by weakening the girth hypothesis to a more general sparsity hypothesis.)  Kang
and Pirot~\cite{KangP} extended this result by proving the same upper bound
when excluding cycles of fewer lengths.
They further strengthened this result~\cite{KangP2} by showing it suffices to
forbid a single cycle length.

\begin{theorem}[\cite{KangP2}]
Let $k$ be a positive integer and $\ell$ be an even positive integer such that
$\ell\ge 2k+2$.  The supremum of $\chi^k(G)$, over all graphs $G$ with
maximum degree $\Delta$ and no cycles of length $\ell$
is $\Theta(\Delta^k/\log k)$, as $\Delta\to \infty$.
\end{theorem}

Finally, we remark briefly about coloring exact distance graphs.
The \emph{exact distance-$k$ graph} $G^{[\natural k]}$ has as its vertex set
$V(G)$.  Two vertices are adjacent in $G^{[\natural k]}$ precisely if they are
at distance exactly $k$ in $G$.
We will not formally define graph classes with \emph{bounded expansion}, but
examples of such classes include all graphs embeddable in any fixed surface and,
more generally, all graphs with any fixed graph $H$ forbidden as a minor.
{\nesetril} and Ossona de Mendez~\cite[Theorem~11.8]{NOdM_sparsity-book} proved the
following.
\begin{theorem}[\cite{NOdM_sparsity-book}]
Let $\G$ be a class of graphs with bounded expansion.
\begin{enumerate}
\item If $k$ is an odd positive integer, then there exists a constant $C_1$ (as a
function of $\G$ and $k$) such that for every graph $G\in\G$ we have   
$\chi(G^{[\natural k]})\le C_1$.
\item If $k$ is an even positive integer, then there exists a constant $C_2$ (as a
function of $\G$ and $k$) such that for every graph $G\in\G$ we have   
$\chi(G^{[\natural k]})\le C_2\Delta(G)$.
\end{enumerate}
\end{theorem}

The values of $C_1$ and $C_2$ arising from the proofs in~\cite{NOdM_sparsity-book}
are very large.  These values have been significantly reduced by subsequent work
of (among others) Zhu~\cite{Zhu-exact}, Stavropoulos~\cite{stavropoulos}, and 
van den Heuvel, Kierstead, and Quiroz~\cite{vdHKQ}.

\section{Open Problems and Conjectures}
In this short section, for easy reference we collect some problems
and conjectures from throughout this survey that remain open.  For consistency,
we phrase each as a conjecture, although a few were initially only posed as
questions, which we indicate where relevant.  In each instance,
we provide a brief statement of the problem, possibly a few comments, and a link
to the place where the problem or conjecture first appears in this survey.

\label{open-problem-list}
\begin{enumerate}
\item
For all integers $k\ge 1$ and $D\ge 3$, let $\chi^k(D)$ and $\omega^k(D)$
denote, respectively, the maximums over all graphs $G$ with $\Delta\le D$ of
$\chi(G^k)$ and $\omega(G^k)$.  For all $k$ and $D$ we have $\chi^k(D)=\omega^k(D)$.
This is called \hyperref[wegner-conj]{Wegner's Conjecture}.  It is trivially
true for $k=1$.  For $k=2$, it is proved for $D\in \{2,3,4,5,7\}$.

\item
For each integer $t$, there exists a constant $\Delta_t$ such that
all $G$ with $\Delta\ge \Delta_t$ satisfy $\omega^2(G)\le \Delta^2-t$.
This conclusion can be strengthened further to 
(i) $\chi^2(G)\le \Delta^2-t$,
(ii) $\chil^2(G)\le \Delta^2-t$,
(iii) $\chip^2(G)\le \Delta^2-t$,
and even 
(iv) $\chiAT^2(G)\le \Delta^2-t$.
This is Question~\ref{question1}.

\item
Every planar graph $G$ with maximum degree $\Delta(G)\ge 8$ satisfies
$\chi^2(G)\le \left\lfloor \frac32\Delta\right\rfloor+1$.
This is \hyperref[wegner-planar-conj]{Wegner's Planar Graph Conjecture} 
(which also conjectures sharp bounds on $\chi^2(G)$ for planar graphs with
smaller maximum degree).  
This conjecture is known to hold asymptotically, even for list coloring (see
Theorem~\ref{HvdHMR-thm}): For each $\epsilon>0$ there exists
$\Delta_{\epsilon}$ such that if $G$ is a planar graph with $\DeltaG\ge
\Delta_{\epsilon}$, then $\chil^2(G)\le\frac32\DeltaG(1+\epsilon)$.

\item
If $\G$ is a minor-closed class of graphs that excludes $K_{3,t}$ for some
integer $t$, then there exists $\Delta_0$ such that
if $G\in \G$ and $\Delta(G)\ge \Delta_0$, then $\chi^2(G)\le \chil^2(G)\le
\left\lfloor\frac32\Delta(G)\right\rfloor+1$.
This is Conjecture~\ref{wegner-generalization-conj}, due to Havet, van den
Heuvel, McDiarmid, and Reed.  This conjecture is best possible, since equality
holds for the planar graphs in Figure~\ref{fig:wegner}. 

\item
If $G$ is planar and subcubic, then $\chi^2(G)\le 6$ when either (i) $G$ has no 
3-cycle~\cite{DvorakST08}, $G$ can be drawn in the plane with no 5-face~\cite{HJT}, 
or (iii) $G$ is bipartite~\cite{DvorakST08,FHS}.  Obviously, part (iii) is implied by
each of parts (i) and (ii), but even part (iii) remains wide open.
This is Conjecture~\ref{subcubic-planar6-conj}, which combines conjectures from 3 
different papers.

\item
There exists an integer $D$ such that every graph $G$ with $\Delta(G)\ge D$ 
that is 2-degenerate satisfies $\chi^2(G)\le \frac52\Delta(G)$.
More generally,
there exists an integer $D$ such that every graph $G$ with $\Delta(G)\ge D$ and
$\mad(G)<4$ satisfies $\chi^2(G)\le \frac52\Delta(G)$.
These are Questions~\ref{HKP-question1} and~\ref{HKP-question2}.
They are both best possible due to the construction in Theorem~\ref{HKP-thm}.
%As evidence that both questions have affirmative answers, Cranston and Yu~\cite{CY}
%showed that $\omega^2(G)\le \frac52\Delta(G)+72$ when $G$ is 2-degenerate, and
%$\omega^2(G)\le \frac52\Delta(G)+532$ when $\mad(G)<4$.
%Subsequently, Kim and Lian~\cite{KL} strengthened the first of these bounds,
%proving that $\omega^2(G)\le \frac52\Delta(G)$ when $G$ is 2-degenerate and 
%$\Delta(G)$ is sufficiently large.  This motivates the following question,
As weaker form of the second question, we have Conjecture~\ref{question3}: 
There exists an integer $D$ such that every graph $G$ with $\mad(G)<4$ and 
$\Delta(G)\ge D$ satisfies $\omega^2(G)\le \frac52\Delta(G)$.

\item Every graph $G$ with maximum degree $\Delta$ satisfies
$\chi^s(G)\le 1.25\Delta^2$.  This is the
\hyperref[erdos-nesetril-conj]{\erdos--\nesetril~Conjecture}.  The best bound for
general $\Delta$ is $\chi^s(G)\le 1.772 \Delta^2$; see the
\hyperref[1.772-anchor]{end of Section~\ref{erdos-nesetril-sec}}.

\item If $G$ has no cycle on $2k$ vertices, then $\omega^s(G)\le (2k-1)(\Delta-k+1)$.
(Recall that $\omega^s(G)$ is the clique number of the line graph of $G$.)
This bounds is best possible, as explained in the text.
This is Conjecture~\ref{CvBKP-conj}.

\item If $G$ is bipartite with parts $A$ and $B$ and maximum degrees
(in these parts) $\Delta_A$ and $\Delta_B$, then
$\chi^s(G)\le\Delta_A\Delta_B$.  This is Conjecture~\ref{BQm-conj}, 
due to Brualdi and Quinn Massey.

\item 
If $G$ has maximum degree $\Delta$ and has no 5-cycle, then
$\chi^s(G)\le\Delta^2$.  This is Conjecture~\ref{mahdian-conj}, due
to Mahdian.

\item Let $G$ be claw-free.  Now $\chi^2(G)\le \frac54\omega(G)^2$ when
$\omega(G)$ is even and $\chi^2(G)\le (5\omega(G)^2-2\omega(G)+1)/4$
when $\omega(G)$ is odd.  This is Conjecture~\ref{claw-free-conj}, which generalizes 
the \hyperref[erdos-nesetril-conj]{\erdos--\nesetril~Conjecture}.  

\item  If $G$ is subcubic and bipartite and $G$ has no 4-cycle, then
$\chi^s(G)\le 7$.  This is Conjecture~\ref{FSGT-conj}(e), due to
Faudree, Schelp, Gy\'arf\'as, and Tuza. 

\item Let $G$ be a subcubic bridgeless graph.  If $G$ is neither the Wagner
graph nor the graph formed from $K_{3,3}$ by subdiving an edge, then
$\chi^s(G)\le 9$.  This is Conjecture~\ref{subcubic-stronger-conj}(a), due to
Hocquard, Lajou, and Lu\v{z}ar.

\item Let $G$ be a subcubic graph.  If $G$ is bridgeless and $|V(G)|\ge 13$,
then $\chi^s(G)\le 8$.  If $G$ has girth at least 5, then $\chi^s(G)\le 7$.
These are Conjecture~\ref{subcubic-stronger-conj}(b,c), due to Lu\v{z}ar, Ma\v{c}ajov\'{a},
\v{S}koviera, and Sot\'{a}k.
The latter implies the conjecture above of Faudree, Schelp, Gy\'arf\'as, and Tuza. 

\item
There exists a constant $C$ such that if $G$ is any planar graph with girth
$g\ge 5$ and maximum degree $\Delta$, then
$\chi^s(G)\le \left\lceil\frac{2g(\Delta-1)}{g-1}\right\rceil+C$.
This is Conjecture~\ref{HLSS-conj}, due to Hud\'{a}k, Lu\v{z}ar, Sot\'{a}k, and
\v{S}krekovski.

\item Every graph $G$ satisfies $\chi''(G)\le \Delta+2$.
This is Conjecture~\ref{conj:total}, posed independently by Bezhad and
by Vizing.  Asymptotically, the best result on this problem is that there exists
a constant $C$ such that $\chi''(G)\le \Delta+C$; see
Theorem~\ref{thm:MolloyR-total}.

\item If graph $G$ has maximum degree $\Delta$, then $\lambda^{2,1}(G)\le
\Delta^2$.  This is the \hyperref[conj:L21]{$L(2,1)$-Labeling Conjecture}, due
to Griggs and Yeh.  It has been proved for $\Delta$ sufficiently large; see
Theorem~\ref{HRS12-thm}.

\item For every positive integer $h$, there exists a constant
$D_h$ such that $\chi^{h-pcf}(G)\le h\Delta+1$ for all graphs $G$ with $\Delta(G)\ge D_h$.
(Here $\chi^{h-pcf}$ is the minimum integer $k$ that admits a proper $k$-coloring such that
each vertex $v$ has at least $\min\{h,d(v)\}$ colors that appear in $N(v)$ exactly once.)
This is Conjecture~\ref{CDOP-conj}.

\item If $\Delta(G)=3$, then $\chi'_{inj}(G)\le 6$.  
And if $G$ is also bipartite, then $\chi'_{inj}(G)\le 5$.
This is Conjecture~\ref{chi-inj-conj}.

\item Does there exist a constant $C$ such that for every bipartite graph $G$ with 
maximum degree $\Delta$ sufficiently large and each edge list assignment $L$ with 
$|L(e)|\ge C\Delta\ln \Delta$ for all $e\in E(G)$ we always have an injective $L$-edge-coloring?
This is Question~\ref{chilinj-ques}.

\item
Fix integers $k\ge 3$ and $\Delta\ge 3$.  For all but finitely many connected 
graphs $G$ with maximum degree $\Delta$, we have $\chil^k(G)\le
\sum_{i=1}^k\Delta(\Delta-1)^{i-1}+1-k$.  The same bound holds for $\chip^k(G)$.
This would be a list coloring, or painting, analogue of
Theorem~\ref{pierron-thm}.  (In fact, both assertions were only posed as
questions.)

\end{enumerate}

\section*{Acknowledgments}
Thanks to Stephen Hartke, Marthe Bonamy, Wouter Cames van Batenburg, Gexin Yu, 
and two anonymous referees for suggestions that improved the coverage and presentation.

%\footnotesize{
%\bibliographystyle{habbrv-ejc}
\bibliographystyle{habbrv}
\bibliography{GraphColoring}
%
% contents below copied from .bbl generated by two lines above

%}

\end{document}